\documentclass[12pt]{amsart}
\usepackage{amsmath, amssymb}
\usepackage[alphabetic]{amsrefs}
\usepackage{amsfonts, epsfig,color,wrapfig,hyperref}
\usepackage{ xypic,verbatim,amscd,color}
\usepackage{graphicx}
\usepackage{colordvi}
\numberwithin{equation}{section}
\newcommand\lL{\mathfrak{\theta}}
\newcommand\GLL{\operatorname{GL}}

\newcommand\agr{\mathcal{G}r}

\newcommand\bull{\sssize{\bullet}}

\newcommand\tr{\tilde{r}}

\newcommand\mvv\Bbb

\newcommand\frg{\mathfrak{g}}
\newcommand\frh{\mathfrak{h}}

\newcommand\op{\operatorname}

\newcommand\mf{\mathcal{F}}

\newcommand\pone{\Bbb{P}^1}

\newcommand\me{\mathcal{E}}

\newcommand{\git}{\ensuremath{\operatorname{/\!\!/}}}

\newcommand\Fl{\operatorname{Fl}}

\newcommand\mv{\mathcal{V}}
\newcommand\tensor{\otimes}
\newcommand\ml{\mathcal{L}}

\newcommand\mg{\frg}

\newcommand\mh{\mathcal{H}}

\newcommand\mw{\mathcal{W}}

\newcommand\mq{\mathcal{Q}}
\newcommand\rk{\operatorname{rk}}

\newcommand\xpp{x}

\newcommand{\leto}[1]{\stackrel{#1}{\to}}

\newcommand\mgg{\mathcal{G}}

\newcommand\SLL{\operatorname{SL}}

\newcommand\mpp{{\operatorname{ S}}}

\newcommand\vr{\Bbb{V}_{\sL_{r+1},\vec{\lambda},\ell}}
\newcommand\vl{\Bbb{V}_{\sL_{\ell+1},\vec{\lambda}^T,r}}

\newcommand\clar{\Bbb{A}_{\sL_{r+1},\vec{\lambda}}}

\newcommand\aclar{{A}_{\sL_{r+1},\vec{\lambda}}}

\newcommand\parbun{\mathcal{P}{ar}}
\newcommand{\ovop}[1]{\overline{\operatorname{#1}}}

\newcommand{\sL}{\mathfrak{sl}}

\newtheorem{theorem}{Theorem}[section]

\newtheorem{remark}[theorem]{ Remark}

\newtheorem{corollary}[theorem]{Corollary}
\newtheorem{question}[theorem]{Question}
\newtheorem{proposition}[theorem]{Proposition}

\newtheorem{lemma}[theorem]{Lemma}
\newtheorem{clm}[theorem]{Claim}
\newtheorem{definition}[theorem]{Definition}
\newtheorem{definition/lemma}[theorem]{Definition/Lemma}
\newtheorem{defi}[theorem]{Definition}

\newtheorem{example}[theorem]{\bf Example}

\begin{document}

\title[Nonvanishing]{Nonvanishing of conformal blocks divisors on $\ovop{M}_{0,n}$}
\author{Prakash Belkale, Angela Gibney and Swarnava Mukhopadhyay}

\maketitle

\begin{abstract} We introduce and study the problem of finding necessary and sufficient conditions under which a conformal blocks divisor on $\ovop{M}_{0,n}$ is nonzero.
We give necessary conditions in type $\op{A}$, which are sufficient when theta and critical levels coincide.  We show that divisors are subject to
additive identities,
dependent on ranks of the underlying bundle.  These identities amplify vanishing and nonvanishing results and have other applications.
\end{abstract}
\section{Introduction}
From the data of a simple Lie algebra $\mg$, a positive integer $\ell$, and an $n$-tuple $\vec{\lambda}$, of dominant integral weights for $\mg$ at level $\ell$, one can use representation theory to construct on
 the moduli spaces $\ovop{M}_{0,n}$,  of stable $n$-pointed rational curves, a vector bundle $\mathbb{V}_{\mg,\vec{\lambda},\ell}$.  The fiber  $\mathbb{V}_{\mg,\vec{\lambda},\ell}|_{x}$ over a point $x \in \ovop{M}_{0,n}$ is isomorphic to a vector space of conformal blocks, and  $\mathbb{V}_{\mg,\vec{\lambda},\ell}$ is referred to as a vector bundle of conformal blocks \cite{TUY} (also see \cite{Fakh,sorger, Tsuchimoto}).

 The vector bundles $\mathbb{V}_{\mg,\vec{\lambda},\ell}$ are globally generated: There is a surjection
$$\mathbb{A}_{\mg, \vec{\lambda}} =\op{A}_{\mg, \vec{\lambda}}  \times  \ \ovop{M}_{0,\op{n}}\twoheadrightarrow \mathbb{V}_{\mg,\vec{\lambda},\ell},  \ \,  \ \mbox{ where }    \op{A}_{\mg, \vec{\lambda}}=(\tensor_{i=1}^n V_{\lambda_i})_{\frg}=\frac{\tensor_{i=1}^n V_{\lambda_i}}{\frg (\tensor_{i=1}^n V_{\lambda_i})}$$
is the vector space of coinvariants, the largest quotient space on which $\mg$ acts trivially.  Here $V_{\lambda}$ is the irreducible finite dimensional representation of $\frg$ with highest weight $\lambda$. This  gives rise to a morphism  $f_{\mathbb{V}}$ from $\ovop{M}_{0,\op{n}}$ to the Grassmannian of $\op{rk}\mathbb{V}_{\mg,\vec{\lambda},\ell}$ quotients of $\op{A}_{\mg, \vec{\lambda}}$:
 \begin{equation}\label{mapG}
 \overline{\operatorname{M}}_{0,\op{n}} \overset{f_{\mathbb{V}}}{\longrightarrow}  \operatorname{Grass}^{quo}(\op{rk}\mathbb{V}_{\mg,\vec{\lambda},\ell},\op{A}_{\mg, \vec{\lambda}}),
 \ \ \ x \mapsto  \ \ ( \op{A}_{\mg, \vec{\lambda}}\twoheadrightarrow \mathbb{V}_{\mg,\vec{\lambda},\ell}|_{x}).
 \end{equation}

 The conformal blocks divisor  $\mathbb{D}_{\mg,\vec{\lambda},\ell}=c_1(\mathbb{V}_{\frg,\vec{\lambda},\ell})$, is then responsible for the composition of  $f_{\mathbb{V}}$ with the Pl\"ucker embedding $p$ of
 the Grassmannian  $\operatorname{Grass}^{quo}(\op{rk}\mathbb{V}_{\mg,\vec{\lambda},\ell},\op{A}_{\mg, \vec{\lambda}})$ into $\Bbb{P}=\mathbb{P}^{N-1}$, where
 $N={{\op{dim} \op{A}_{\mg,\vec{\lambda}}}\choose{\op{rk}\mathbb{V}_{\mg,\vec{\lambda}}}}$.   In particular, when $\op{dim} \op{A}_{\mg, \vec{\lambda}}=\op{rk}\mathbb{V}_{\mg,\vec{\lambda},\ell}$, the morphism $p\circ f_{\mathbb{V}}$ contracts $\ovop{M}_{0,n}$ to a point, and the conformal blocks divisor $\mathbb{D}_{\mg,\vec{\lambda},\ell}$ is zero. We denote $p\circ f_{\mathbb{V}}$ by $\phi_{\Bbb{D}}$.

The Verlinde formula gives the rank of
 $\vr$  \cite[Thm 4.2.2]{sorger}, and the dimension of the vector space of coinvariants can be computed, allowing one to compare these values.  Using other approaches in type A, in \cite{BGMA}, we show when  $\ell$ surpasses either the {\em{critical}} or {\em{theta levels}} associated to a pair $(\sL_{r+1},\vec{\lambda})$, then the conformal blocks divisor $\mathbb{D}_{\mg,\vec{\lambda},\ell}$ vanishes.

 However, even for low $n$, and for $\mg=\sL_{r+1}$, for $r$ very small, there are many examples (first found by Fakhruddin,  using \cite{ConfBlocks}) where $0<\op{rk}\mathbb{V}_{\mg,\vec{\lambda},\ell}<\op{rk}\mathbb{A}_{\mg, \vec{\lambda}}$, while  the divisor $\mathbb{D}_{\mg,\vec{\lambda},\ell}$  is zero.  For example, the rank of $\mathbb{V}_{\sL_4, \{\omega_1, (2\omega_1+\omega_3)^3 \}, 3}$ on $\ovop{M}_{0,4}$ is one, while the dimension of the vector space of coinvariants
$\mathbb{A}_{\sL_4, \{\omega_1, (2\omega_1+\omega_3)^3 \}}$ is $2$.  A calculation shows that $\mathbb{D}_{\sL_4, \{\omega_1, (2\omega_1+\omega_3)^3 \}, 3}=0$.    In this case, as we explain in Section \ref{Intro}, one can decompose the divisor (in this case, a point) as the following sum
 $$\mathbb{D}_{\sL_4, \{\omega_1, 2\omega_1+\omega_3, 2\omega_1+\omega_3, 2\omega_1+\omega_3 \}, 3}=\mathbb{D}_{\sL_4, \{\omega_1, \ldots, \omega_1 \}, 1} + \mathbb{D}_{\sL_4, \{0, \omega_1+\omega_3,  \omega_1+\omega_3,  \omega_1+\omega_3 \}, 2}.$$
Both of the divisors on the right hand side turn out to be trivial for simple reasons.

\medskip
This has led us to ask the following:

\begin{question}\label{Question}
What are necessary and sufficient conditions for  a triple $(\mg,\vec{\lambda},\ell)$ that guarantee that the associated conformal blocks divisor  $\mathbb{D}_{\mg,\vec{\lambda},\ell}$ is nonzero?
\end{question}
Determining when $\op{rk}\mathbb{V}_{\mg,\vec{\lambda},\ell}$ (resp. $\op{rk}\mathbb{A}_{\mg, \vec{\lambda}})$
is nonzero is a classical problem in representation theory \cite{fulton,b4}. Frequently, there is an inductive structure, indicating that
nonzeroness is controlled by smaller dimensional Lie algebras (in type A). One could ask if there are similar inductive non-vanishing results for the  first (and higher) Chern classes of $\mathbb{V}_{\mg,\vec{\lambda},\ell}$; Question \ref{Question} belongs to this line of inquiry (see Lemma \ref{chernvanish}). We also note that this question does not seem to be amenable
to exact numerical formulas \cite{Fakh} for the Chern classes (and ranks cf. \cite{sorger}).

As described earlier, conformal blocks divisors are base point free and so give rise to morphisms from $\ovop{M}_{0,n}$ to other projective varieties.  In order to study the maps they define, one is interested in which curves the divisors contract.  The question of vanishing is at its heart a practical matter, first to establish whether the map associated to a given divisor is trivial, and then if not trivial, to understand its image.
Recent groundbreaking work \cite{MDSBS, GoKa} has shown that $\ovop{M}_{0,n}$ is not a Mori Dream Space for $n\ge 13$, and begs one to rethink the long held hope that  there may be only a finite number of extremal rays of $\op{Nef}(\ovop{M}_{0,n})$.  Question \ref{Question} relates to this problem, since there are potentially an infinite number of distinct nonzero conformal blocks divisors that span extremal rays of the nef cone (also see Remark \ref{remarkF}).

\medskip
{\em{Overview:}}
\begin{enumerate}

\medskip

\item We simplify Question \ref{Question}  for divisors of any type by  decomposition of weights and level.  In particular, we show that if certain rank conditions are satisfied, conformal blocks divisor classes are subject to additive relations (Proposition \ref{additive}). Combining this with a quantum generalization of Fulton's conjecture, we prove a scaling identity for divisors associated to rank one bundles (Corollary \ref{LevelOneMaps}), amplifying our nonvanishing results. In other applications, we exhibit non-trivial conformal blocks divisors on $\ovop{M}_{0,n}$ for all $n \ge 5$, with non-zero weights, that do not give rise to a birational morphism (Section \ref{odd}), and we show how mysterious vanishing of a particular divisor can be explained by mundane facts about its constituents (as in the examples in Section \ref{Intro}).
\smallskip

\item It was shown in \cite{BGMA}  that in type A, divisors $\mathbb{D}_{\sL_{r+1}, \vec{\lambda},\ell}$ vanish if $\ell$ is above the so-called critical or theta levels (described in Defs \ref{CritLevDefi} and \ref{FSVCriticalLevel}).   The critical level is defined only for type A, while the theta level is defined for divisors in all types.   Here we prove that in all types, divisors vanish if $\ell$ is above the theta level  (Lemma \ref{theta_level_vanishing}).
\smallskip

\item  We simplify Question \ref{Question} for divisors in type $\op{A}$ by  decomposition of the Lie algebra.
 In particular, the action of the $\sL_2$ corresponding to the highest root gives  a relationship between $\mathbb{V}_{\sL_{r+1},\vec{\lambda},\ell}$  with a tensor product of conformal block bundles for
Lie algebras of smaller rank.  In Theorems \ref{mon1} and \ref{mon2}, we give sufficient nonvanishing conditions for $\mathbb{D}_{\sL_{r+1},\vec{\lambda},\ell}$,  which we show are also necessary when the critical and theta levels coincide and are equal to $\ell$.  We do this by first proving a stronger statement for $\sL_2$, where we answer the question completely (Corollary \ref{metric}). We then use
the $\sL_2$ statement and an argument using parabolic bundles, to obtain results for $\sL_{r+1}$.

\smallskip
\item We apply our methods to show nonvanishing of two infinite families of critical level divisors.  For the family $\mathbb{D}_{\sL_{r+1},\omega_1^n,\ell}$, $n=(r+1)(\ell+1)$, we characterize all  boundary curves that are contracted,  showing that if the $S_n$-invariant F-Conjecture holds, then these divisors give the reduction morphisms to certain moduli spaces $\ovop{M}_{0,\mathcal{A}}$ defined by Hassett (Section \ref{HassettFamily}).  For a related family, we show that while the associated morphisms factor through maps to Hassett spaces, they don't necessarily give them (Section \ref{Kostka}). \end{enumerate}

\medskip
{\em{To describe our results more precisely, we set a small amount of notation.}}

\subsubsection{Notation}

For a finite dimensional simple Lie algebra $\frg$, and a positive integer  $\ell$, called the level, let ${P}_{\ell}(\frg)$ denote the set of dominant integral weights $\lambda$ with $(\lambda,\theta)\leq \ell$.  Here $\theta$ is the highest root, and $(\ ,\ )$ is the Killing form, normalized so that $(\theta,\theta)=2$.  To a triple $(\frg,\vec{\lambda},\ell)$, such that $\vec{\lambda} \in P_{\ell}(\frg)^n$, there  corresponds a vector bundle $\mathbb{V}_{\frg,\vec{\lambda},\ell}$ of conformal blocks  on the stack $\overline{\mathcal{M}}_{g,n}$  \cite{TUY, sorger}.

Finite dimensional irreducible polynomial representations for $\operatorname{GL}_{r+1}$ are parameterized by Young diagrams $\lambda=(\lambda^{(1)}\geq \lambda^{(2)}\geq\dots\geq\lambda^{(r)}\geq \lambda^{(r+1)}\geq 0)$.  We note that two Young diagrams $\lambda$ and $\mu$ give the same representation of $\operatorname{SL}_{r+1}$ (equivalently $\sL_{r+1}$) if $\lambda^{(a)}-\mu^{(a)}$ is a constant independent of $a$.
We use the notation $|\lambda|=\sum_{i=1}^r \lambda^{(i)}$, and $\lambda\in P_{\ell}(\sL_{r+1})$ if and only if $\lambda^{(1)}-\lambda^{(r+1)}\leq \ell$.  We will say that $\lambda$ is normalized if $\lambda^{(r+1)}=0$. The normalization of $\lambda$ is the Young diagram
$\lambda-\lambda^{(r+1)}\cdot(1,1,\dots,1)$.

\medskip
\subsection{Additive identities dependent on ranks}

 In Section \ref{tensorr}, we give the following criteria for decomposing a divisor as an effective sum of simpler conformal blocks divisors.
\begin{proposition}\label{Additive} Given $\vec{\mu} \in P_{\ell}(\frg)^n$, and $\vec{\nu}\in P_{m}(\frg)^n$  with $\rk \mathbb{V}_{\frg,\vec{\mu},\ell}=1,$
and $\rk \mathbb{V}_{\frg,\vec{\mu}+\vec{\nu},\ell+m}= \rk \mathbb{V}_{\frg,\vec{\nu},m}=\delta$.  Then  $$\mathbb{D}_{\frg,\vec{\mu}+\vec{\nu},\ell+m}= \delta\cdot\mathbb{D}_{\frg,\vec{\mu},\ell}+\mathbb{D}_{\frg,\vec{\nu},m}.$$
\end{proposition}

\subsubsection{Applications}  Using Proposition \ref{Additive} in conjunction with the quantum generalization of a conjecture of Fulton in invariant theory \cite{fultonconjecture} and \cite[Remark 8.5]{qdeform}, we show  in Corollary  \ref{fulton} that if $\rk\vr=1$, then
$$\Bbb{D}_{\sL_{r+1},N\vec{\lambda},N\ell}=N\cdot \Bbb{D}_{\sL_{r+1},\vec{\lambda},\ell},  \ \ \forall N \in \mathbb{N}.$$
  As an application,  in Corollary \ref{LevelOneMaps}, we identify images of the maps $\phi_{\mathbb{D}}$  for $\mathbb{D}=\mathbb{D}_{\sL_{r+1}, \ell \vec{\lambda},\ell}=\ell\cdot \mathbb{D}_{\sL_{r+1}, \vec{\lambda},1}$,
as the generalized Veronese quotients of \cite{Giansiracusa,GJM}.

Proposition \ref{Additive} can be used to show that a divisor is nontrivial, by writing it as an effective sum of simpler divisors, and then showing one of the
summands is nontrivial.  In  Section ~\ref{odd}, we use Proposition \ref{Additive} to give non-trivial conformal blocks divisors, with non-zero weights, that do not give  birational morphisms. Such examples were not known before.
 One may also approach questions of mysterious vanishing in this way, seeing for example a divisor as a sum of divisors whose vanishing can be explained by other means (cf. Section \ref{Intro}).

\smallskip

\subsection{Vanishing above the theta level in all types}  In Lemma \ref{theta_level_vanishing}, we show that  $\mathbb{D}_{\mg, \vec{\lambda},\ell}$ is zero if $\ell$ is above the theta level.  It was shown in \cite{BGMA}  that in type A, divisors $\mathbb{D}_{\sL_{r+1}, \vec{\lambda},\ell}$ vanish if  $\ell$ surpasses the critical or theta levels. The critical level (Def \ref{CritLevDefi}),
comes from the relationship with conformal blocks in type A to quantum cohomology, and is valid only for divisors type A.  The theta level (Def \ref{FSVCriticalLevel}), comes from the interpretation of a vector space of conformal blocks as an explicit quotient \cite{Beauville,FSV2}, and holds in all types.

\begin{defi}\label{CritLevDefi} \cite{BGMA} Let $\vec{\lambda}=(\lambda_1,\dots,\lambda_n)$ be an $n$-tuple of  {\em normalized} integral weights for $\sL_{r+1}$, assume that $r+1$ divides $\sum_{i=1}^n|\lambda_i|$, and define the critical level for the pair  $(\sL_{r+1},\vec{\lambda})$ to be
$$ c(\sL_{r+1},\vec{\lambda})=-1+\frac{1}{r+1}\sum_{i=1}^n|\lambda_i|.$$
One can define $c(\sL_{r+1},\vec{\lambda})$ in general, by replacing each $\lambda_i$ by its normalization.   One refers to $\mathbb{V}_{\sL_{r+1},\vec{\lambda},\ell}$  as a critical level bundle
 if $\ell= c(\sL_{r+1},\vec{\lambda})$ and $\vec{\lambda} \in P_{\ell}(\sL_{r+1})^n$.
\end{defi}

  \begin{defi}\label{FSVCriticalLevel} \cite{BGMA} Given a pair $(\frg, \vec{\lambda})$, one refers to
$$\lL(\frg,\vec{\lambda})= -1+ \frac{1}{2} \sum_{i=1}^n\lambda_i (H_{\theta})\in \frac{1}{2}\Bbb{Z}$$
as the theta level.  Here $H_{\theta}$ is the co-root corresponding to the highest root $\theta$.
\end{defi}

  \subsection{Necessary and sufficient conditions for non-vanishing of divisors in type $\op{A}$}
In Theorem \ref{mon1} we answer Question \ref{Question}  in case the critical and theta levels coincide and are equal to
$\ell$.   To state the result, we describe how to associate two
{\em{auxiliary bundles}} to $\mathbb{V}_{\sL_{r+1}, \vec{\lambda},\ell}$.

\begin{definition}(Auxiliary bundles) Given
$\vr$,  one forms the  bundles
$\mathbb{V}_{\sL_2,\vec{\mu},\ell}$, and $\mathbb{V}_{\sL_{r-1},\vec{\nu},\ell}$, where for $i \in [n]=\{1,\ldots,n\}$,
 $\mu_i$ correspond to the $2\times \ell$ Young diagrams formed by the first and last rows of $\lambda_i$, and $\nu_i$
 are given by the $(r-1)\times\ell$ diagram obtained by removing the first and last rows of $\lambda_i$. Note that $\nu_i$ may not be normalized, and can be zero.
\end{definition}

\begin{theorem}\label{mon1}Let $\vec{\lambda}  \in {P}_{\ell}(\sL_{r+1})^n$, for $n\ge 4$, and suppose that $\mathbb{V}_{\sL_2,\vec{\mu},\ell}$, and $\mathbb{V}_{\sL_{r-1},\vec{\nu},\ell}$  are the
auxiliary bundles constructed above.  If
\begin{enumerate}
\item $\ell=c(\sL_{r+1},\vec{\lambda})=\theta(\sL_{r+1},\vec{\lambda})$;
\item $\lambda_i \ne 0$ for all $i\in[n]=\{1,\ldots,n\}$, normalized; and
\item $\operatorname{rk} \mathbb{V}_{\sL_2,\vec{\mu},\ell} > 0$,
 \end{enumerate}
then $\mathbb{D}_{\sL_{r+1},\vec{\lambda},\ell} \ne 0  \iff \mbox{ for } r\ge 3, \operatorname{rk} \mathbb{V}_{\sL_{r-1},\vec{\nu},\ell} > 0.$
\end{theorem}

The forward implication of Theorem \ref{mon1} is a special case of Theorem \ref{mon2}, which provides a very general non-vanishing statement.  In particular, Theorem \ref{mon2} does not assume that the theta level equals the critical level, or that the level $\ell$ is critical.

We give applications of Theorem \ref{mon1} and Theorem \ref{mon2}, using them to show families of divisors are nontrivial as we next explain.

 \subsection{Nonvanishing of two families of critical level divisors}
In Section \ref{HassettFamily}, using Theorem \ref{mon2}, and our results from \cite{BGMA}, we  give a complete description of all boundary curves contracted by the nontrivial critical level divisors
$\Bbb{D}=\mathbb{D}_{\sL_{r+1}, \omega_1^{n},\ell}$ on
  $\overline{\operatorname{M}}_{0,n}$,  where $n=(r+1)(\ell+1)$, for all $r>1$ and $\ell>1$.

  In Section \ref{Kostka}, we apply Theorem \ref{mon1} and \cite{BGMA}, to show that  divisors of the form
$\Bbb{D}_{\vec{\alpha}} = \mathbb{D}_{\sL_{r+1}, \{\alpha_i\omega_1\}_{i=1}^{n},\ell}$, for $\sum_{i}\alpha_i=(r+1)(\ell+1)$,  with $n=2(r+1)$ are nonzero.

   These two families of critical level divisors are generalizations of those studied previously for $r=1$ and $\ell=1$ ($\mathbb{D}_{\sL_{2}, \omega_1^{2(\ell+1)},\ell}$ in \cite{Fakh, gjms} and $\mathbb{D}_{\sL_{r+1}, \omega_1^{2(r+1)},1}$ in \cite{Fakh,Giansiracusa}).

   We show that when $r>1$ and $\ell>1$, in the first family $\mathbb{D}_{\sL_{r+1}, \omega_1^{n},\ell}$, where $n=(r+1)(\ell+1)$, the divisors contract  exactly the same curves as maps $\rho_{\mathcal{A}}$ to Hassett's moduli spaces $\ovop{M}_{0,\mathcal{A}}$, with weights $\mathcal{A}=(\frac{1}{\ell+r},\ldots,\frac{1}{\ell+r})$ (Corollary \ref{Complete}).  Hassett spaces are defined in Section \ref{Hassett}.   In particular, as we explain, if  the $S_n$ invariant F- conjecture holds, then the normalization of the image of their associated maps will  be isomorphic to the Hassett spaces $\ovop{M}_{0,\mathcal{A}}$.

 We show in Section \ref{Kostka}, that while the  morphisms given by the divisors $\mathbb{D}_{\sL_{r+1}, \{\alpha_i\omega_1\}^n_i,\ell}$ factor through Hassett spaces $\ovop{M}_{0,\mathcal{A}}$, where $\mathcal{A}=(\frac{\alpha_1}{\ell+r},\ldots,\frac{\alpha_n}{\ell+r})$, it is possible to exhibit vectors $\vec{\alpha}=(\alpha_1,\ldots,\alpha_n)$ so that $\mathbb{D}_{\sL_{r+1}, \{\alpha_i\omega_1\}^n_i,\ell}$ contracts more curves than $\rho_{\mathcal{A}}$.   In particular, the images of the maps given by $\mathbb{D}_{\sL_{r+1}, \{\alpha_i\omega_1\}^n_i,\ell}$ are not, in general, isomorphic to $\ovop{M}_{0,\mathcal{A}}$.


\subsection{A note on our methods} \medskip
The main results of this paper are proved by using identifications of conformal blocks to generalized theta functions \cite{pauly} and a version of the quantum generalization of the Horn conjecture \cite{b4}. The applications use standard intersection theoretic computations on $\ovop{M}_{0,\op{n}}$,  the factorization formulas of \cite{TUY}, and  applications of the quantum cohomology of Grassmannians to conformal blocks \cite{witten} (also \cite{b4} and \cite{BGMA}). We recommend the article of Sorger ~\cite{sorger} for a concise primer on the definition of conformal blocks.

It is an interesting question whether our results can be proved by numerical formulas for the first Chern classes \cite{Fakh} and ranks
(cf. \cite{sorger}) of conformal block bundles. We have not been able to do so because of difficulties with factorization data and
Casimir values (also see \cite{BGMA}).

\subsection{Acknowledgements}
We thank N. Fakhruddin, D. Jensen, A. Kazanova, S. Kumar, and  H-B. Moon for useful discussions.

 In \cite{Fakh}, Fakhruddin gives explicit formulas for the classes $c_1 (\mathbb{V}_{\mg,\vec{\lambda},\ell})$, and the intersections of $c_1 (\mathbb{V}_{\mg,\vec{\lambda},\ell})$ with F-Curves. These formulas have been implemented by Swinarski in Macaulay2 \cite{ConfBlocks}, which we have used to do many examples.

P.B. and S.M.  were supported on NSF grant  DMS-0901249, and A.G. on DMS-1201268.

\section{Additive identities between divisors when rank conditions are satisfied}\label{tensorr}

In this section we prove Proposition \ref{Additive} which shows that when certain rank conditions are satisfied, then divisors decompose as sums reflecting the decomposition of their weights and levels into sums.  This result enables one to simplify questions of vanishing of a particular divisor into problems about its simpler constituent parts.  But there have been more applications as well.  For example,
in \cite{Kaz} this result was used to prove that any $\op{S}_n$-invariant divisor for $\sL_n$ on $\ovop{M}_{0,n}$ coming from a bundle of rank one was in fact a sum of level one divisors in type $A$.  In particular, the cone generated by infinitely many such divisors is finitely generated.

\subsection{Behavior under tensor products}
Let $G$ be a simple, simply connected algebraic group with Borel subgroup $B$ and Lie algebra $\frg$. Let $\hat{\frg}$ denote the corresponding affine Lie algebra (see e.g., \cite{sorger}). For a dominant integral weight $\lambda$  in $P_{\ell}(\frg)$, let $V_{\lambda}$ denote the corresponding finite dimensional representation of $\frg$ with highest weight $\lambda$. Let $\mathcal{H}_{\lambda,\ell}$ denote the corresponding irreducible representation of $\hat{\frg}$. Note that $V_{\lambda}\subseteq \mathcal{H}_{\lambda,\ell}$ (we simply write $\mathcal{H}_{\lambda}$ when the level $\ell$ is clear from the context).

As is explained for example in \cite{manon}, if $\lambda$ and $\nu$ are dominant integral weights in
$P_{\ell}(\frg)$ and ${P}_{m}(\frg)$,
 then one has a canonical inclusion, mapping highest weight vectors to tensor products of highest weight vectors
 $\mh_{\mu+\nu,\ell+m}\subseteq \mh_{\mu,\ell}\tensor\mh_{\nu,m}$,
which restricts to a natural inclusion
$V_{\mu+\nu}\subseteq V_{\mu}\tensor V_{\nu}$.

Suppose $\vec{\mu}=(\mu_1,\dots,\mu_n)$ and $\vec{\nu}=(\nu_1,\dots,\nu_n)$ are $n$-tuples of dominant integral weights in $P_{\ell}(\frg)$ and ${P}_{m}(\frg)$.   There is a natural diagram of vector bundles on $\overline{\operatorname{M}}_{0,n}$, with surjective vertical arrows (cf. \cite{manon}).
\begin{equation}\label{sibel}
\xymatrix{
\Bbb{A}_{\frg,\vec{\mu}+\vec{\nu}}\ar[r]\ar[d] & \Bbb{A}_{\frg,\vec{\mu}}\tensor \Bbb{A}_{\frg,\vec{\nu}}\ar[d]\\
 \mathbb{V}_{\frg,\vec{\mu}+\vec{\nu},\ell+m}\ar[r]^-{\phi} & \mathbb{V}_{\frg,\vec{\mu},\ell}\tensor \mathbb{V}_{\frg,\vec{\nu},m}.
}
\end{equation}

Now suppose
\begin{equation}\label{fassump}
\rk \mathbb{V}_{\frg,\vec{\mu},\ell}=1.
\end{equation}
Then we claim
\begin{proposition}\label{additive}
The map $\phi$ is a surjection. If
\begin{equation}\label{rankassumption}
\rk \mathbb{V}_{\frg,\vec{\mu}+\vec{\nu},\ell+m}= \rk \mathbb{V}_{\frg,\vec{\nu},m}=\delta,
\end{equation}
then
the map $\phi$ is an isomorphism, and hence
$$\mathbb{D}_{\frg,\vec{\mu}+\vec{\nu},\ell+m}= \delta\cdot\mathbb{D}_{\frg,\vec{\mu},\ell}+\mathbb{D}_{\frg,\vec{\nu},m}.$$
\end{proposition}
\begin{proof}
We will assume \eqref{rankassumption} and show that the dual map of $\phi$ is an isomorphism (below we show that the dual map is always injective fiber wise under the assumption \eqref{fassump}).

Let $y$ be an arbitrary closed point of $\overline{\operatorname{M}}_{0,n}$. Let $u$ and $v$ be non-zero elements of $\mathbb{V}_{\frg,\vec{\mu},\ell}|_y^*$ and $\mathbb{V}_{\frg,\vec{\nu},m}|_y^*$  respectively.
Let $\bar{u}$ and $\bar{v}$ denote their (non-zero) images in  $\Bbb{A}_{\frg,\vec{\mu}}|_y^*$ and $\Bbb{A}_{\frg,\vec{\nu}}|_y^*$ respectively.
We want to show that the image of $\bar{u}\tensor\bar{v}$ in $\Bbb{A}_{\frg,\vec{\mu}+\vec{\nu}}|_y^*$ is non-zero.  
 It suffices (using \eqref{fassump}) to prove the following (classical) statement: If $\alpha$ and $\beta$ are non-zero elements
in ${A}_{\frg,\vec{\mu}}^*$ and ${A}_{\frg,\vec{\nu}}^*$ respectively, then the image of $\alpha\tensor\beta$ under
$${A}_{\frg,\vec{\mu}}^*\tensor {A}_{\frg,\vec{\nu}}^*\to  {A}_{\frg,\vec{\mu}+\vec{\nu}}^*$$
is non-zero.
Write  commutative diagrams for each $i$
\begin{equation}\label{sibel3}
\xymatrix{
G/B    \ar[r]^-{\Delta}\ar[d] & G/B\times G/B\ar[d]\\
 \Bbb{P}(V_{\mu_i+\nu_i})\ar[r]& \Bbb{P}(V_{\mu_i}\tensor V_{\nu_i})
}
\end{equation}
For every dominant integral weight $\lambda$ there is a line bundle $\ml_{\lambda}$ on $G/B$ whose global sections equal
$V_{\lambda}^*$ ($\ml_{\lambda}$ is the pull back of $\mathcal{O}(1)$ via the map $G/B\to \Bbb{P}(V_{\lambda})$).
Note that $X=(G/B)^n$ carries a diagonal action of $G$. Define the following line bundle for every $\vec{\lambda}$:
$\ml_{\vec{\lambda}}=\boxtimes_{i=1}^n\ml_{\lambda_i}$.
Note that  ${A}_{\frg,\vec{\lambda}}^*=H^0(X,\ml_{\vec{\lambda}})^G$. Under the multiplication map (induced by $n$ fold product of the diagram \eqref{sibel3}),
$H^0(X,\ml_{\vec{\mu}})^G\times H^0(X,\ml_{\vec{\nu}})^G\to H^0(X,\ml_{\vec{\mu}+\vec{\nu}})^G$
, the image of $\alpha\times \beta$ is non-zero (because locally we are forming the product of non-zero functions on $X$). This implies the
  desired non-vanishing.
\end{proof}

\subsection{First application: scaling for divisors associated to rank one bundles}
Fulton conjectured that if $\rk\clar=1$ then  $\rk \Bbb{A}_{\sL_{r+1},N\vec{\lambda}}=1$ $\forall$ $N\in \mathbb{Z}_{>0}$. This was proved by Knutson, Tao and Woodward \cite{KTW2}.  The quantum generalization of Fulton's conjecture \cite{fultonconjecture,qdeform} is the following: Suppose $\rk\vr=1$ ($\ell$ is not necessarily the critical level) then $\rk\Bbb{V}_{\sL_{r+1},N\vec{\lambda},N\ell}=1$ for all positive integers $N$. Using this generalization and Proposition \ref{additive}, we obtain (by induction):
\begin{corollary}\label{fulton}
If $\rk\vr=1$, then $\Bbb{D}_{\sL_{r+1},N\vec{\lambda},N\ell}=N\cdot \Bbb{D}_{\sL_{r+1},\vec{\lambda},\ell}$,  $\forall$ $N\in \mathbb{Z}_{>0}$.
\end{corollary}

\begin{remark} Corollary \ref{fulton} appears in case  $r=1$ and $\vec{\lambda}=(\omega_1,\ldots,\omega_1)$ in \cite{gjms}. An analogous result for $\frg=\mathfrak{so}_{2r+1}$ appears in ~\cite{SM1}.
\end{remark}
\begin{corollary}\label{LevelOneMaps}  Let $\mathbb{D}=\mathbb{D}_{\sL_{r+1}, \ell \vec{\lambda},\ell}$ be a nontrivial conformal blocks divisor, so that $\sum_{i}|\lambda_i|=(r+1)(d+1)$.
The image of the morphism $\phi_{\mathbb{D}}$ is isomorphic to the Veronese quotient $U_{d,n}\git_{(0,\mathcal{A})}\operatorname{SL}(d+1)$, where $a_i=|\lambda_i|/(r+1)$.
\end{corollary}

\begin{proof}
By assumption, $\lambda_1,\dots,\lambda_n\in P_1(\sL_{r+1})$ and $\rk \mathbb{V}_{\sL_{r+1},\ell\vec{\lambda},\ell}\neq 0$.
Therefore, by  Proposition \ref{hornp}, $\operatorname{rk}\mathbb{V}_{\sL_{r+1},\vec{\lambda},1}\ne 0$.  Using factorization, $\operatorname{rk}\mathbb{V}_{\sL_{r+1},\vec{\lambda},1}\ne 0$ implies that
$\operatorname{rk}\mathbb{V}_{\sL_{r+1},\vec{\lambda},1}=1$ \cite{Fakh}.  So by Corollary \ref{fulton},
$\Bbb{D}_{\sL_{r+1},\ell\vec{\lambda},\ell}=\ell \cdot \Bbb{D}_{\sL_{r+1},\vec{\lambda},1}$ for all positive integers $\ell$.  Applying \cite[Theorem 1.2]{GiansiracusaGibney} gives the assertion.
\end{proof}

\subsection{Second application: fibrations given by conformal blocks divisors whose weights are all nonzero}\label{Fibration}
Nonzero divisors $\mathbb{D}=\mathbb{D}_{\frg,\vec{\lambda},\ell}$ with (some) zero weights give rise to maps with images that are not birational to $\ovop{M}_{0,n}$.   In many known examples of conformal blocks divisors, where all weights are nonzero, the images of associated maps are birational to $\ovop{M}_{0,n}$.  However, this is not always so, as the  examples below show.

\subsection{Examples of non  birational morphisms}\label{odd}
Let  $n=2k+1$, for $k \ge 2$, and $\frg=\sL_{2k+1}$. We define the following bundles:
\begin{enumerate}
\item For $\vec{\mu}=(\underset{k\ factors}{\omega_1,\ldots,\omega_1},\underset{k\ factors}{\omega_{2k},\ldots,\omega_{2k}},0)$ at level $\ell=1$. We know that  $\mathbb{D}_{\sL_{2k+1}, \vec{\mu}, 1} \ne 0$, since this is the pullback of a nonzero bundle from $\ovop{M}_{0,2k}$ \cite[Proposition 2.4, (1)]{Fakh}.
\item For $\vec{\nu}=(\omega_1,\ldots, \omega_{1},\omega_{1})$, consider $\mathbb{D}_{\sL_{2k+1}, \vec{\nu}, m}$ at level $m=1$.
\item For  $\vec{\mu}+\vec{\nu}= (\underset{k}{2\omega_1,\ldots,2\omega_1},\underset{k}{\omega_1+\omega_{2k},\ldots,\omega_1+\omega_{2k}}, \omega_{1})$, we consider  $\mathbb{D}_{\sL_{2k+1},\vec{\mu}+\vec{\nu}, \ell+m}$, at level $\ell+m=2$.
\end{enumerate}
One must check that   $\mathbb{V}_{\sL_{2k+1},\vec{\mu}+\vec{\nu}, \ell+m}$  has rank $1$, so that Proposition \ref{additive} is applicable, and
$$\mathbb{D}_{\frg,\vec{\mu}+\vec{\nu},\ell+m}= \delta\cdot\mathbb{D}_{\frg,\vec{\mu},\ell}+\mathbb{D}_{\frg,\vec{\nu},m},$$
and here $\delta =1$, since bundles in type $A$ of level one have rank one.

To check the rank of $\mathbb{V}_{\sL_{2k+1},\vec{\mu}+\vec{\nu}, \ell+m}$   is one, we use Witten's theorem relating quantum cohomology and conformal blocks (see e.g., \cite{BGMA}), which says that the rank of $\mathbb{V}_{\sL_{2k+1},\vec{\mu}+\vec{\nu}, \ell+m}$   is equal to the coefficient of the class of $q^{k-1}\sigma_{2 \omega_{2k+1}}$ in the quantum product:
$$\sigma_{2\omega_1}^{\star (k)} \star \sigma_{\omega_1 + \omega_{2k}}^{\star (k)} \star \sigma_{\omega_1} \star \sigma_{2\omega_1}^{\star (k-1)} \in \op{QH}^{\star}(\op{Gr}(2k+1,2k+1+2)).$$
(the standard notation of cycle classes of Schubert varieties, as well as the definition of quantum cohomology appear for example in \cite{BGMA})

Now, by quantum Pieri, $\sigma_{\omega_1 + \omega_{2k}}=\sigma_{2\omega_1}\star\sigma_{\omega_{2k-1}}$ and
 $\sigma_{2\omega_1}^{\star (2k)}\star  \sigma_{\omega_1}= \sigma_{\omega_{2k+1}+\omega_{2k}}$.

Our coefficient is therefore the same as the coefficient of
the class of
of $q^{k-1}\sigma_{2 \omega_{2k+1}}$ in the quantum product:
$$ \sigma_{\omega_{2k-1}}^{\star (k)} \star \sigma_{\omega_{2k+1}+\omega_{2k}} \star \sigma_{2\omega_1}^{\star (k-1)}\in \op{QH}^{\star}(\op{Gr}(2k+1,2k+1+2))$$
which again by Witten's theorem is the rank of the conformal block $\mathbb{V}_{\sL_{2k+1},\vec{\gamma}, 2}$ with
$\gamma_1=\dots=\gamma_k=\omega_{2k-1}$ and $\gamma_{k+1}=\omega_{2k+1}+\omega_{2k}$. By \cite{Beauville}, we may dualize to find the ranks, so our
rank is the same as the rank of the conformal block $\mathbb{V}_{\sL_{2k+1},\vec{\gamma}^*, 2}$. But $\gamma^*$ is the collection of weights
$(\omega_2)^k$ and $\omega_{1}$. This critical level for the block $\mathbb{V}_{\sL_{2k+1},\vec{\gamma}^*,2}$ is zero and we can compute the rank of the block at level $1$ as the space of classical coinvariants. The rank is therefore $1$ by Pieri.


But $\mathbb{D}_{\frg,\vec{\nu},m}=0$, since the critical level is zero (see  \cite{Fakh}). Moreover, since
$\mathbb{D}_{\frg,\vec{\mu},\ell}$ is pulled back from $\overline{\operatorname{M}}_{0,2k}$,  $\mathbb{D}_{\frg,\vec{\mu}+\vec{\nu},\ell+m}$ is a non-trivial conformal blocks divisor on  $\overline{\operatorname{M}}_{0,2k+1}$ pulled back from
$\overline{\operatorname{M}}_{0,2k}$, and hence does not correspond to a birational map, but rather a fibration. We note that this could not practically be checked by computer for $k$ beyond $4$.

\subsection{Third application: Using decomposition to explain vanishing}\label{Intro}

As was mentioned in the introduction, Proposition \ref{additive} explains the vanishing of the critical level divisor $\mathbb{D}_{\sL_4, \{\omega_1, (2\omega_1+\omega_3)^3 \}, 3}$.  A calculation shows that $\op{rk}\mathbb{V}_{\sL_4, \{\omega_1, (2\omega_1+\omega_3)^3 \}, 3}=1$, while the coinvariants have rank $2$. The map \eqref{mapG} maps to a positive dimensional Grassmannian, and so one does not expect the first Chern class to be zero.

Noting that by a calculation,
$$\op{rk} \mathbb{V}_{\sL_4, \{0, (\omega_1+\omega_3)^3 \}, 2}=\op{rk}\mathbb{V}_{\sL_4, \{\omega_1, \ldots, \omega_1 \}, 1}=1,$$
one can write this divisor as a sum
$$\mathbb{D}_{\sL_4, \{\omega_1, (2\omega_1+\omega_3)^3 \}, 3}=\mathbb{D}_{\sL_4, \{\omega_1, \ldots, \omega_1 \}, 1} + \mathbb{D}_{\sL_4, \{0, (\omega_1+\omega_3)^3 \}, 2}.$$ But since $\mathbb{D}_{\sL_4, \{\omega_1, \ldots, \omega_1 \}, 1}$ is trivial since the sum of the areas of the weights is $4$, and $\mathbb{D}_{\sL_4, \{0, (\omega_1+\omega_3)^3 \}, 2}$ is trivial,  being pulled back from $\ovop{M}_{0,3}$, one sees why $\mathbb{D}_{\sL_4, \{\omega_1, (2\omega_1+\omega_3)^3 \}, 3}$ is trivial as well.

\medskip

One can write down other similar examples.  For instance, while $$1=\op{rk}\mathbb{V}_{\sL_4, \{\omega_2+\omega_3, \omega_1, \omega_1+2\omega_2, 2\omega_1+\omega_3 \}, 3}<\op{rk}\mathbb{A}_{\sL_4, \{\omega_2+\omega_3, \omega_1, \omega_1+2\omega_2, 2\omega_1+\omega_3 \}} = 2,$$ one can argue  that  the critical level divisor $\mathbb{D}_{\sL_4, \{\omega_2+\omega_3, \omega_1, \omega_1+2\omega_2, 2\omega_1+\omega_3 \}, 3}$, which is dual to itself under critical level duality, is a sum of three divisors  pulled back from $\ovop{M}_{0,3}$:
$$\mathbb{D}_{\sL_4, \{\omega_2+\omega_3, \omega_1, \omega_1+2\omega_2, 2\omega_1+\omega_3 \}, 3}=\mathbb{D}_{\sL_4, \{\omega_2, 0,\omega_1, \omega_1 \}, 1}
+ \mathbb{D}_{\sL_4, \{\omega_3, 0, \omega_2, \omega_3\}, 1} + \mathbb{D}_{\sL_4, \{0, \omega_1, \omega_2, \omega_1\}, 1}.$$

\section{Decomposition with respect to the Lie Algebra}

\subsection{Nonvanishing of conformal blocks divisors for $\sL_2$}

In this section we prove Corollary \ref{metric} which is a key step in our nonvanishing results.  The first part will be used in the proof of non-vanishing criteria for  conformal blocks divisors in Theorems \ref{mon1} and \ref{mon2}.  The second part says that sub-critical level conformal blocks divisors for $\sL_2$ are non-zero as long as their ranks are not equal to zero (compare with \cite{Fakh}, and B. Alexeev's formula \cite[(3.5)]{Swin})).


\subsubsection{Generalities}\label{Set}
Let  $x=(z_1,\dots,z_n)$ be an $n$-tuple of distinct points in $\Bbb{A}^1\subseteq \pone$. Set $A={A}_{\frg,\vec{\lambda}}$ and denote by
 $C_{x}$, the image of the map $T_{x}^{\ell+1}:W\to W$, where $W=V_{\lambda_1}\tensor\dots\tensor V_{\lambda_n}$ and
$T_{x}=\sum_{i=1}^n z_i e^{(i)}_{\theta}$ with $e^{(i)}_{\theta}$ acting on the $i$th coordinate.

Then by  \cite[Proposition 4.1]{Beauville} and \cite[Section 1.1]{FSV2},
\begin{lemma}The fiber of $\mathbb{V}_{\frg,\vec{\lambda},\ell}$ at $x$ is the cokernel of the natural map $C_{x} \to A$.
\end{lemma}
This immediately yields vanishing statements for the theta level:
\begin{lemma}\label{theta_level_vanishing}
Suppose that $\ell>\theta(\frg,\vec{\lambda})$, then   $\Bbb{D}_{\frg,\vec{\lambda},\ell}=0$.
\end{lemma}
\begin{proof}
Assume $\ell>\theta(\frg,\vec{\lambda})$, or that $\sum\lambda_i(H_{\theta})< 2 (\ell+1)$. Writing elements of $V_{\lambda_i}$ as obtained from lowest weight vectors by application of the operators $e_{\alpha}$, we see that any element of $C_x$ is a sum of eigenvectors  for $H_{\theta}$ with  strictly positive eigenvalues, and hence  maps to zero under $C_x\to A$.
Therefore $C_x\to A$ is the zero map. This proves the lemma.
\end{proof}


\subsubsection{The fixed part of conformal blocks and vanishing of $\mathbb{D}_{\frg,\vec{\lambda},\ell}$}

 \begin{definition}\label{Fixed} Consider the ``fixed part'':
$$\mathbb{F}(\frg,\ell,\vec{\lambda})=\bigcap_{x\in \operatorname{M}_{0,n}} \mathbb{V}_{\frg,\vec{\lambda},\ell}\mid_x^* \ \subseteq A^*=\mathbb{A}_{\frg,\vec{\lambda}}|_x^*.$$
\end{definition}
The intersection in Definition \ref{Fixed} is the same as if we were to intersect over all points $x\in \overline{\operatorname{M}}_{0,n}$. We also characterize $\mathbb{F}(\frg,\ell,\vec{\lambda})$ as the space of global sections (over $\overline{\operatorname{M}}_{0,n}$) of $\mathbb{V}_{\frg,\vec{\lambda},\ell}^*$: $\mathbb{F}(\frg,\ell,\vec{\lambda})=H^0( \overline{\operatorname{M}}_{0,n},\mathbb{V}_{\frg,\vec{\lambda},\ell}^*)$.
\begin{definition}
Let $C'\subseteq V_{\lambda_1}\tensor\dots\tensor V_{\lambda_n}$ be the $\Bbb{C}$-linear span of elements of the form
\begin{equation}\label{expresso}
e_{\theta}^{a_1}v_1\tensor\dots\tensor e_{\theta}^{a_{n}}v_n, \ v_i\in V_{\lambda_i},\ 0\leq a_i\leq \ell+1,\ \sum_{i=1}^n a_i =\ell+1.
\end{equation}
\end{definition}
\begin{lemma}
\begin{enumerate}
\item $\mathbb{F}(\frg,\ell,\vec{\lambda})^*$ is the cokernel of the natural map $C'\to A$.
\item $c_1\mathbb{V}_{\frg,\vec{\lambda},\ell}=0$ if and only if $C'$ and $C_x$ have the same image in $A$. $C_x$ was defined in the beginning of this section,  $C_x\subseteq C'$.
\end{enumerate}
\end{lemma}
\begin{proof}
An element $\alpha\in A^*$ is in the fixed part, if and only if $\alpha(T_x^{\ell+1}(v_1\tensor\dots\tensor v_n))=0$ as a polynomial in $z_1,\dots,z_n$, where $v_i$ are arbitrary elements of $V_{\lambda_i}$. This polynomial is zero if  all its coefficients are zero.  So $\alpha\in \mathbb{F}(\frg,\ell,\vec{\lambda})$ if and only if $\alpha(C')=0$ as desired.
This gives (1). It is easy to see that (2) follows from (1).
\end{proof}

\begin{corollary}\label{metric}
Suppose $\frg=\sL_2$, and $\ell$ the critical level for $\vec{\lambda}$. Suppose $\tilde{\ell}\leq \ell$ and $\vec{\lambda}$ is in ${P}_{\tilde{\ell}}(\sL_{2})$.
\begin{enumerate}
\item $\mathbb{F}(\frg,\tilde{\ell},\vec{\lambda})=0$.
\item If  $\rk \mathbb{V}_{\frg,\tilde{\ell},\vec{\lambda}}\neq 0$, then $c_1\mathbb{V}_{\frg,\tilde{\ell},\vec{\lambda}}\neq 0$.
\end{enumerate}
\end{corollary}
\begin{proof}
To prove (1) it suffices to consider the case $\tilde{\ell}=\ell$.
Let $C''$ be the set of $\frh$-invariants  of $V_{\lambda_1}\tensor\dots \tensor V_{\lambda_n}$. It is easy to see that $C''$ surjects on to $A$. We will
show $C''\subseteq C'$ and hence prove (1). A tensor $\gamma$ in $C''$ can be written as sum of vectors of  the form $e_{\theta}^{a_1}v_1\tensor\dots\tensor e_{\theta}^{a_{n}}v_n$ with $v_i$ lowest weight vectors. Since $\gamma$ is $\frh$-invariant,
we should have (in each term) $2\sum a_i -\sum\lambda_i=0$, so $\sum a_i=\ell+1$. Therefore $\gamma\in C'$. This gives (1). Now (2) follows from the lemma above and (1), since the map $f_{\Bbb{V}}$ from \eqref{mapG} is non-constant (if $X$ is a  positive dimensional projective variety, $f:X\to \Bbb{P}^m$ a non-constant morphism, then $f^*\mathcal{O}(1)$ is a non-trivial line bundle on $X$).
\end{proof}

\begin{lemma}\label{chernvanish}
Suppose $V$ is a globally generated vector bundle on a projective variety $X$. The following are equivalent:
\begin{enumerate}
\item[(1)] $V$ is a trivial vector bundle, i.e., isomorphic to $\mathcal{O}_X^{\oplus r}$, $r=\rk V$.
\item[(2)] The Chern character $\operatorname{ch}(V)$ of $V$ equals $\op{rk} V\in H^0(X,\Bbb{Q})\subseteq H^*(X,\Bbb{Q})$.
\item[(3)] $c_1(V)=0\in H^2(X,\Bbb{Q})$.
\end{enumerate}
\end{lemma}
\begin{proof}
Note that the $H^0(X,\Bbb{Q})$ and $H^2(X,\Bbb{Q})$ components of  $\operatorname{ch}(V)$ are $\op{rk} V$ and $c_1(V)$ respectively. We therefore only need to show that (3) implies (1). Assume (1) fails. We may assume $X$ to be connected and positive dimensional. Let $A\tensor \mathcal{O}_X\to V$ be the surjection from a constant vector bundle to $V$. Let  $f: X\to \operatorname{Grass}^{quo}(\op{rk} V ,A)$ be the corresponding map to the Grassmann variety of quotients. Since $f$ is a non-constant morphism (otherwise $V$ would be trivial as a vector bundle), $c_1(V)=c_1(f^*\mathcal{O}(1))\neq 0\in H^2(X,\Bbb{Q}).$ (One may use a smooth projective curve $C$  and a map $g:C\to X$ such that $f\circ g$ is not constant to show that the degree of the pull back of $\mathcal{O}(1)$ to $C$ is non-zero.)
\end{proof}

\subsection{Non-vanishing criteria: Proof of Theorem \ref{mon2} and first implication of Theorem \ref{mon1}}\label{nonzero}\label{mon2proof}

In this section we prove Theorem \ref{mon2} which also gives the proof of the forward implication in statement of Theorem \ref{mon1} as a special case.   The proof of Theorem \ref{mon1} is completed next in Section \ref{Conversemon1}.

\bigskip

To state the result, we begin with a more general construction of auxiliary bundles.

\begin{definition}\label{Auxiliary}(More general auxiliary bundles) Given $\vec{\lambda}\in P_{\ell}(\sL_{r+1})^n,$ such that for each $i \in [n]$, $\lambda_i$ is normalized.   For each $i\in[n]$, choose a two element
subset $A_i=\{\alpha_i<\beta_i\}\subseteq [r+1]$. Consider associated conformal blocks bundles $\Bbb{V}_{\sL_{2},\vec{\mu},\ell}$ and $\Bbb{V}_{\sL_{r-1},\vec{\nu},\ell}$
where $\mu_i$ is the $2\times \ell$ Young diagram formed by the $\alpha_i$th and $\beta_i$th rows of $\lambda_i$, and $\nu_i$ is the $(r-1)\times\ell$ Young diagram formed by removing the $\alpha_i$th and $\beta_i$th rows of $\lambda_i$, $i\in [n]$. The $\mu_i$ and $\nu_i$ may not be normalized.
\end{definition}

\begin{theorem}\label{mon2} Given $\vec{\lambda}\in P_{\ell}(\sL_{r+1})^n,$ such that for each $i$, $\lambda_i$ is normalized. Suppose that:
\begin{enumerate}
\item[(a)] $\frac{1}{2}\sum_{i=1}^n |\mu_i|=\frac{1}{r-1}\sum_{i=1}^n |\nu_i|=\frac{1}{r+1}\sum_{i=1}^n |\lambda_i|=\delta\in\Bbb{Z}$
\item[(b)] Assume that $\ell$ is not greater than the critical level for $\vec{\mu}$ (one needs to normalize $\vec{\mu}$ to find the critical level), and
$\rk\Bbb{V}_{\sL_{2},\vec{\mu},\ell}\neq 0$.
\item[(c)] If $r>2$, then $\rk\Bbb{V}_{\sL_{r-1},\vec{\nu},\ell}\neq 0$ (so condition (c) is vacuous for $r=2$).
\end{enumerate}
Then $\mathbb{D}_{\sL_{r+1},\vec{\lambda},\ell} \ne 0$.
\end{theorem}

\subsection{Key Steps}
We first assume that $r-1\geq 2$, and we will later indicate the modifications required to treat the case $r=2$. We now outline the main steps:




\subsubsection{Step I: Maps of conformal blocks}
We note that  given the auxiliary bundles as described in Definition \ref{Auxiliary}, one can form the diagram
\begin{equation}\label{sibel2}
\xymatrix{
\Bbb{A}_{\sL_2,\vec{\mu}}\tensor \Bbb{A}_{\sL_{r-1},\vec{\nu}}\ar[rr]\ar[d] &  & \Bbb{A}_{\sL_{r+1},\vec{\lambda}}\ar[d]\\
 \mathbb{V}_{\sL_2,\vec{\mu},\ell}\tensor \mathbb{V}_{\sL_{r-1},\vec{\nu},\ell}\ar[rr]^-{\phi} & & \mathbb{V}_{\sL_{r+1},\vec{\lambda}}
}
\end{equation}


\subsubsection{Step II} Next we show, in Section  \ref{cherryblossom}, that $\phi$ is a generically non-zero map over $\overline{\operatorname{M}}_{0,n}$ (and thereby proving also that $\rk \mathbb{V}_{\sL_{r+1},\vec{\lambda}}\neq 0$).
This step uses the geometry of parabolic bundles over a point of ${\operatorname{M}}_{0,n}$.

\subsubsection{Step III} Finally we argue by contradiction, suppose $c_1\mathbb{V}_{\sL_{r+1},\vec{\lambda}}
=0$. Then  it will follow that  $\mathbb{V}_{\sL_{r+1},\vec{\lambda},\ell}$ is trivial as a vector bundle. The image of  $\vr\mid^*_x$ in $\Bbb{A}\mid^*_x=A^*$ with  $\Bbb{A}=\Bbb{A}_{\sL_2,\vec{\mu}}\tensor \Bbb{A}_{\sL_{r-1},\vec{\nu}}$ (the constant fibers of $\Bbb{A}$ are denoted by $A$) is a constant  non-zero subspace which lies inside the image of the dual of $\mathbb{V}_{\sL_2,\vec{\mu},\ell}\mid_x\tensor \Bbb{A}_{\sL_{r-1},\vec{\nu}}$ inside $A^*$. But, this
contradicts \ref{metric}, which implies that
$$\bigcap_{x\in {\operatorname{M}}_{0,n}}\mathbb{V}_{\sL_2,\vec{\mu},\ell}\mid^*_x=0\subseteq A_{\sL_2,\vec{\mu}}^*.$$
 Therefore $c_1\mathbb{V}_{\sL_{r+1},\vec{\lambda}}
\neq 0$, as desired.

\subsubsection{Some notation}\label{azure}
\begin{enumerate}
\item Let $\agr_{r+1}$ be the affine Grassmannian of rank $ r+1$-vector bundles  with trivialized determinants on $\pone$, and  trivialized outside of $p$. $\agr_{r+1}$ is identified with the (ind-variety) $\SLL_{r+1}(\Bbb{C}((z)))/\SLL_{r+1}(\Bbb{C}[[z]])$ where $z$ is a local coordinate at $p$.
\item A quasi-parabolic $\operatorname{SL}_{r+1}$ bundle on $\pone$ is a triple $(\mv,\mf,\gamma)$ where  $\mv$ is a vector bundle on $\pone$ of  rank $r+1$ and degree $0$ with a given trivialization $\gamma:\det\mv\leto{\sim}\mathcal{O}$, and  $\mf=(F^{p_1}_{\bull},\dots, F^{p_n}_{\bull})\in \Fl_S(\mv)$ is a collection of complete flag  on fibers over $p_1,\dots,p_n$ (see Definition \ref{classic}). Let $\parbun_{r+1}$ be the moduli stack parameterizing   quasi-parabolic $\operatorname{SL}_{r+1}$ vector bundles on $\pone$.

\end{enumerate}

\subsubsection{Weyl group translates of highest weight vectors}
Let $S=\Bbb{C}^{r+1}$ with basis vectors $\epsilon_1,\dots,\epsilon_{r+1}$ and dual basis $L_1,\dots,L_{r+1}$. Let $U=\Bbb{C}\epsilon_1\oplus \Bbb{C}\epsilon_2$ and $W= \Bbb{C}\epsilon_3\oplus\dots\oplus\Bbb{C}\epsilon_{r+1}$ so that one has an internal direct sum $S=U +W$. There is a natural map
$\GLL(U)\times\GLL(W)\to \GLL(S)$. Identify $\GLL(U)=\GLL(2), \GLL(W)=\GLL(r-1)$ and $\GLL(S)=\GLL(r+1)$ in the evident way. Let $\frh_U$, $\frh_W$ and $\frh_S$
be the Cartan algebras of $\sL_2$, $\sL_{r-1}$ and $\sL_{r+1}$ respectively. The Weyl group $S_{r+1}$ of  $\sL_{r+1}$ can be considered to be a subgroup of $\GLL(S)$ (as permutation matrices), and acts on $\frh_S$ and $\frh_S^*$: $\pi\in S_{r+1}$ acts as $\pi \epsilon_i=\epsilon_{\pi(i)}$ and $\pi\cdot L_i=L_{\pi(i)}$:
\begin{itemize}
\item  If $\lambda\in \frh_S^*$ then
$(\pi\cdot\lambda)(\epsilon_i)=\lambda(\epsilon_{\pi^{-1}(i)})$. Therefore $\pi L_i(\epsilon_j)=L_i(\epsilon_{\pi^{-1}(j)})=\delta_{i,\pi^{-1}(j)}$  and hence $\pi L_i=L_{\pi(i)}$.
\end{itemize}

Now let $V_{\lambda}$ be an irreducible representation of $\GLL_{r+1}$ with highest weight vector $v$, and highest weight $\lambda$.  Let $\pi\in S_{r+1}$.
\begin{lemma}
\begin{enumerate}
\item The vector $\pi v$ a is weight vector of weight $\pi\lambda$.
\item $\pi v$ is a highest weight vector of $\sL_2\oplus\sL_{r-1}$ if and only if
\begin{equation}\label{organized}
\pi^{-1}(1)<\pi^{-1}(2),\ \pi^{-1}(3)<\dots<\pi^{-1}(r+1)
\end{equation}
\end{enumerate}
\end{lemma}
\begin{proof}
If $h\in \frh_S$, $h(\pi v)=\pi (\pi^{-1}h\pi) v= \lambda(\pi^{-1} \cdot h )\pi v=(\pi\cdot \lambda)(h)\pi v$. Therefore $\pi v$ is a weight vector, of weight $\pi\lambda$.

Let $e_{ij}\in \sL_{r+1}$ take $\epsilon_j$ to $\epsilon_i$ and all others to zero. Then  $\pi^{-1}e_{ij}\pi\epsilon_{\pi^{-1}(i)}=e_{\pi^{-1}(j)}$ and so $\pi^{-1}e_{ij}\pi=e_{\pi^{-1}(i),\pi^{-1}(j)}$. Moreover, $e_{ij} \pi v =\pi (\pi^{-1}e_{ij}\pi)v=\pi e_{\pi^{-1}(i),\pi^{-1}(j)} v$ which is zero if $\pi^{-1}(i)<\pi^{-1}(j)$.
\end{proof}
Assuming \eqref{organized}, denote the corresponding irreducible representation of $\sL_2\oplus \sL_{r-1}$ by $V_{\mu}\tensor V_{\nu}$. The representation of $\sL_2$ corresponds to  $\pi^{-1}(1)$ and $\pi^{-1}(2)$ rows of $\lambda$ (which gives $\mu$), and the representation corresponding to $\sL_{r-1}$ corresponds to the remaining rows of $\lambda$ (which gives $\nu$): This is because $\pi\lambda(\epsilon_i)=\lambda(\epsilon_{\pi^{-1}(i)})$.


For the final steps we will use the description of conformal blocks by generalized theta functions. We refer the reader to  Appendix \ref{appendo} for more details and (some) standard notation on parabolic moduli stacks, and the geometric
realization of conformal blocks.

\subsubsection{Geometrization of branching}
Let $X_a=\SLL(a)/B_a$, where $B_a$ is a chosen Borel subgroup. There is a natural map
$\iota:X_2\times X_{r-1}\to X_{r+1}$ given by $(g,h)\mapsto gh\pi$. There is a natural map $X_{r+1}\to \Bbb{P}(V_{\lambda})$. The pull backs of $\mathcal{O}(1)$ to $X_2\times X_{r-1}$ and $X_{r+1}$ are
$\ml_{\mu}\boxtimes\ml_{\nu}$ and $\ml_{\lambda}$ respectively. Then (compatibly)
$H^0(X_2\times X_{r-1},\ml_{\mu}\boxtimes\ml_{\nu}) = (V_{\mu}\tensor V_{\nu})^*$
 and $H^0(X_r,\ml_{\lambda}) = V_{\lambda}^*.$

Note further that $\iota$ is the map $\Fl(U)\times\Fl(W)\to \Fl(S)$ given by $(F_{\bull}, G_{\bull})\mapsto H_{\bull}$ where $H_{\bull}$ is computed as
follows: $$H_a= F_m\oplus G_{k},\ m=\pi^{-1}\{1,2\}\cap[i],\ k=a-m.$$

\subsubsection{The final step}\label{cherryblossom}
Working over  $x=(\pone,p_1,\dots,p_n)\in {\operatorname{M}}_{0,n}$, we produce an element $\delta\in \mathbb{V}_{\sL_{r+1},\vec{\lambda}}^*\mid_x$ whose image via $\phi^*$ in $(\mathbb{V}_{\sL_2,\vec{\mu},\ell}\tensor \mathbb{V}_{\sL_{r-1},\vec{\nu},\ell})^*\mid_x$ is non-zero. Therefore $\phi$ is not the zero map.

Consider the maps of moduli stacks
$$\beta:\parbun_{2}\times\parbun_{r-1}\to \parbun_{r+1}$$
($\parbun_{\tr}$ are the moduli stacks from Section \ref{verlindesection} with $n$-marked points $p_1,\dots,p_n$). Here $\beta$ is the map that sends
$(\mv,\mf,\gamma)\times (\mw,\mgg,\gamma')\mapsto (\mv\oplus\mw,\mh,\gamma\wedge\gamma')$ where
$H^{p}_a= F^p_m\oplus G^p_{k}$ where $m$ is the number of elements in $\pi_i^{-1}\{1,2\}$ that are less than or equal to $a$, $k=a-m$.

Consider line bundles  $\mathcal{P}_2= \mathcal{P}( \sL_{2},{\ell},\vec{\mu})$ on $\parbun_2$, $\mathcal{P}_{r-1}= \mathcal{P}( \sL_{r-1},{\ell},\vec{\nu})$ on $\parbun_{r-1}$ and $\mathcal{P}_{r+1}= \mathcal{P}( \sL_{r+1},{\ell},\vec{\lambda})$ on $\parbun_{r+1}$.
\begin{lemma}
The map $\beta$ pulls back $\mathcal{P}_{r+1}$ to $\mathcal{P}_{2}\boxtimes \mathcal{P}_{r-1}$
and induces the dual of the map $\phi$ at the level of global sections.
\end{lemma}
\begin{proof}
Introduce a new point $p\in \pone$. The map $\beta$ is dominated by a map
$$(\agr_2\times \Fl(U)^n) \times (\agr_{r-1}\times \Fl(W)^n)\to (\agr_{r+1}\times \Fl(S)^n).$$

\end{proof}

Therefore our final task can be restated in geometric terms as:
The map
\begin{equation}\label{openeye3}
H^0(\parbun_{r+1},\mathcal{P}_{r+1})\to H^0(\parbun_{2},\mathcal{P}_2)\tensor H^0(\parbun_{r-1},\mathcal{P}_{r-1})
\end{equation}
is non-zero. For this we need a geometric way of producing elements of these spaces. In Section \ref{openeye2}, we recall a way of construction sections:
Write $\ell-D= \frac{1}{r+1}\sum|\lambda_i|$. Consider an evenly split bundle ( see Section \ref{openeye2} for a definition) $\mq$ of degree $-D$ and rank $\ell$. For every $\mgg\in \Fl_S(\mq)$ a section $s_{(\mq,\mgg)}$ is produced in  $H^0(\parbun_{r+1},\mathcal{P}_{r+1})$ by a degeneracy locus construction.

 We will now show that if $\mgg\in \Fl_S(\mq)$ is generic, $s_{(\mq,\mgg)}\in H^0(\parbun_{r+1},\mathcal{P}_{r+1})$ maps to a non zero element under the map \eqref{openeye3}. So we need to show that $s_{(\mq,\mgg)}$ is non-zero on images of generic elements of the form  $(\mv,\mf,\gamma)\times (\mw,\mgg,\gamma')$ via $\beta$.

Suppose not, then we will find maps $\psi_1:\mv\to \mq$ and $\psi_2:\mw\to\mq$, such that the resulting map $\mv\oplus\mw\to\mq$ is non-zero and
$$(\psi_1)_p(F^p_a)\subseteq G^p_{\ell-\mu_i^{(a)}}, (\psi_1)_p(F^p_b)\subseteq G^p_{\ell-\mu_i^{(b)}},\ a\in[2],\  b\in[r-1], p=p_i\in S.$$
But  there are no such non-zero maps by Lemma \ref{closedeye} applied to the non-zero vector spaces from conditions (b) and (c) of the theorem. Here we note that if we fix a $(\mv,\mf)$ in \eqref{openeye}, the zeroness holds for generic $(\mq,\mgg)$.
\begin{remark}
At the point $p=p_i$, $F^p_a$ maps to $H^p_{\alpha_a}$ therefore the requirement is that it map to $G^p_{\ell-\lambda_i^{(\alpha_a)}}$ but
$\lambda_i^{(\alpha_a)}=\mu_i^{(a)}$; similarly for $\nu$.
\end{remark}

\subsubsection{Case $r=2$}
 We just omit the $\sL_{r-1}$ factor. The transversality statement boils down to the following: Let $L$ be a one dimensional complex vector space. Then there are no non-zero maps $\psi$ such that for all
$p=p_i$, $i=1,\dots,n$,
$$\psi:L\tensor\mathcal{O}\to \mathcal{Q},\  \psi_p(L_p)\subseteq G^p_{\ell-\mu_i^{(1)}}.$$

 One can prove this by converting the above transversality assertion into the non-zeroness of a generalized Gromov-Witten number (using an argument of the type used in  Proposition \ref{closedeye}), the fact that the small quantum cohomology ring of a projective space is simply governed by degree constraints, and the shifting operations from \cite{b4}. Here we sketch a more direct argument: If $\psi_p$ are all non-zero, then the above follows from Kleiman's transversality. If some $\psi_p$ are zero, say for $p_1,\dots,p_{m}$ then $\psi$ gives rise to a map $L(\sum_{i\leq m} p_i)\to \mq$, we may apply Kleiman's transversality and find the expected dimension to be negative.

\subsection{Proof of the reverse implication in Theorem \ref{mon1}}\label{Conversemon1}

\begin{proof}(of the reverse implication in Theorem \ref{mon1})
Note that $\pi_i$ are the same permutation $\pi$ here. Let $\frg'=\sL_{\theta}\oplus\sL_{r-1}$ be the $\pi^{-1}$ conjugate embedding of
$\sL_2\oplus \sL_{r-1}$. Let $v_i \in V_{\lambda_i}$ be the highest weight vectors. Break up each $V_{\lambda_i}$ into a direct sum $M_i\bigoplus \oplus_{j\in I_i}W^{i}_j$ of irreducible modules for $\frg'$ where
$M_i$ is the irreducible module with highest weight vector $v_i$. It is easy to see there are no eigenvectors for $h_{\theta}$ with weight
$\lambda_i^{(1)}$ in any of the $W^{i}_{j}$ (because they will involve at least one application of $f_{\alpha_1}$ or $f_{\alpha_r}$, which lower the $h_{\theta}$ weight).

Therefore under the quotient
$T^{\ell+1}_{x}: V_{\lambda_1}\tensor\dots\tensor V_{\lambda_n}\to A_{\sL_{r+1},\vec{\lambda}}$,  the only term that survives
is the image of  $T^{\ell+1}_{x} M$ with $M=M_1\tensor \dots \tensor M_n$. But the image of the coinvariants $M_{\frg'}$ in $\Bbb{V}_{\sL_{r+1},\vec{\lambda},\ell}$ is zero since it factors through $\Bbb{V}_{\sL_{2},\vec{\mu},\ell}\tensor\Bbb{V}_{\sL_{r-1},\vec{\nu},\ell}=0$, from \eqref{sibel2}.  This implies  that the image of $T^{\ell+1}_{x} M$ in $A_{\sL_{r+1},\vec{\lambda}}$ is equal to the image of $M_{\frg'}$, which is constant.

Therefore, $\Bbb{V}_{\sL_{r+1},\vec{\lambda},\ell}$ is a constant quotient of $A_{\sL_{r+1},\vec{\lambda}}$, and has zero first Chern class.

\end{proof}

\section{The family of divisors $\mathbb{D}=\mathbb{D}_{\sL_{r+1},\omega_1^n,\ell}$, $n=(r+1)(\ell+1)$}\label{HassettFamily}

In this section we consider the set of divisors $\mathbb{D}=\mathbb{D}_{\sL_{r+1},\omega_1^n,\ell}$, for $n=(r+1)(\ell+1)$, each of which is $\operatorname{S}_n$-invariant.   In particular, by \cite{KeelMcKernan,GibneyCompositio}
the morphisms $\phi_{\mathbb{D}}$ given by these nef and big divisors are birational.  Propositions 5.3 and 5.4 in  \cite{BGMA} together give a list of $\operatorname{F}$-curves contracted by the divisors $\mathbb{D}_{\sL_{r+1},\omega_1^{n},\ell}$.  By Corollary \ref{Complete},  for  $\mathcal{A}=(\frac{1}{\ell+r},\ldots,\frac{1}{\ell+r})$,  $\ell>1$, and $r>1$,
 the maps $\rho_{\mathcal{A}}$ and $\phi_{\mathbb{D}}$ are shown to contract the same $\operatorname{F}$-curves.   According to the $\op{S}_n$-invariant $\operatorname{F}$-conjecture, the divisors $\mathbb{D}$ and $\rho_{\mathcal{A}}^*(A)$, where $A$ is any ample divisor on $\overline{\operatorname{M}}_{0,\mathcal{A}}$ conjecturally lie on the same face of the nef cone of  $\overline{\operatorname{M}}_{0,n}$.  In particular,  the (normalization of the) image of the morphism $\phi_{\mathbb{D}}$ should be isomorphic to $\overline{\operatorname{M}}_{0,\mathcal{A}}$.

 In \cite{KiemMoon}, and \cite{HBThesis} it is shown that $\overline{\operatorname{M}}_{0,\mathcal{A}}$ can be constructed as a GIT quotient of $\overline{\operatorname{M}}_{0,\mathcal{A}}(\mathbb{P}^1,1)$ by $\SLL(2)$.
The case $\ell=1$, the image of  $\phi_{\mathbb{D}}$ was shown in \cite{Fakh} to be isomorphic to $(\mathbb{P}^1)^n\git_{\mathcal{A}} \operatorname{SL}(2)$, where $a_i=1/(r+1)$.     In case $r=1$,  the image of  $\phi_{\mathbb{D}}$ was shown in \cite{Giansiracusa} to be isomorphic to $U_{\ell,n}\git_{(\delta,\mathcal{A})} \operatorname{SL}(\ell+1)$, where $\delta=\frac{\ell-1}{\ell+1}$, and $a_i=\frac{1}{\ell+1}$ (see \cite{gjms} for this particular notation).

\smallskip
We begin by defining an F-Curve.

\begin{definition}\label{FCurve}Fix a partition of $[n]=\{1,\ldots,n\}$ into four nonempty sets $N_1$, $N_2$, $N_3$, $N_4=[n]\setminus (N_1 \cup N_2 \cup N_3)$, and consider the
 morphism
$$\overline{\operatorname{M}}_{0,4} \longrightarrow \overline{\operatorname{M}}_{0,n},   \ \ \ \ (C, (a_1,a_2,a_3,a_4))\mapsto (X, (p_1,\ldots,p_n))$$
where $X$ is the nodal curve obtained as follows.  If $|N_i| \ge 2$, then one glues a copy of $\mathbb{P}^1$ to the spine $(C, (a_1,a_2,a_3,a_4))$ by attaching a point
$$(\mathbb{P}^1, \{p_j : j \in N_i\} \cup \{\alpha_i\} ) \in M_{0,|N_i|+1}$$
 to $a_i$ at $\alpha_i$.  If $|N_i|=1$, one does not glue any curve at the point $a_i$, but instead labels $a_i$ by $p_i$.    We refer to any element of the numerical equivalence class of the image of this morphism  the $\operatorname{F}$-Curve $F(N_1,N_2,N_3)$ or by $F(N_1,N_2,N_3,N_4)$, depending on the context.
\end{definition}

\begin{proposition}\label{ExactNonzeroResult}\label{e}
 Suppose that $r\geq 1$ and $\ell\geq1$.  For $n=(r+1)(\ell+1)$,  the divisor $\mathbb{D}=\mathbb{D}_{\sL_{r+1},\omega_1^{n},\ell}$,  positively intersects all $\operatorname{F}$-curves $F(N_1,N_2,N_3,N_4)$ with $i \in \{1,2,3,4\}$, $|N_i|=n_i$, where
$n_1\le n_2 \le n_3 \le n_4=(r+1)(\ell+1)-\sum_{1\le i \le 3}n_i$, and $\sum_{i=1}^3n_i \ge r+\ell+1$.
\end{proposition}

\begin{proof}

Our proof carries
a larger induction hypothesis,  and we prove a stronger statement for these cases.

We want to show that any $F$-curve $F(N_1,N_2,N_3,N_4)$, $|N_i|=n_i$ with $n_i\leq (r+1)(\ell+1)-(r+\ell+1)=r\ell,i=1,\dots,4$ is not contracted by $\mathbb{D}$ (so we drop the hypothesis that $n_1\leq n_2\leq n_3\leq n_4$).   By  \cite{Fakh},
$$\mathbb{D} \cdot F(N_1,N_2,N_3,N_4)= \sum_{\vec{\lambda}=(\lambda_1,\ldots,\lambda_4)\in P_{\ell}^4}\operatorname{deg}\mathbb{V}_{\vec{\lambda}} \ \Pi_{i=1}^4
\operatorname{rk}\mathbb{V}_{\sL_{r+1},(\omega_1^{n_i}, \lambda_i^{*}),\ell},$$
where  $\mathbb{V}_{\vec{\lambda}}=\mathbb{V}_{\sL_{r+1},\vec{\lambda}}$.  This is a sum of nonnegative numbers.  Therefore, to show that the sum is nonzero, it is enough to show that there is at least one element
$\vec{\lambda}=(\lambda_1,\ldots,\lambda_4)\in P_{\ell}^4$ for which
$$\operatorname{deg}\mathbb{V}_{\vec{\lambda}} \ \Pi_{i=1}^4
\operatorname{rk}\mathbb{V}_{\sL_{r+1},(\omega_1^{n_i}, \lambda_i^{*}),\ell} >0.$$

We note that if $\lambda_i$ are normalized dominant integral weights  for $\sL_{r+1}$ in  $P_{\ell}(\sL_{r+1})$ (so they fit into boxes of size  $r\times\ell$) with $|\lambda_i|=n_i$, then
$\rk\mathbb{V}_{\sL_{r+1},(\omega_1^{n_i}, \lambda_i^{*}),\ell} >0$, since this classical, and we may use the Pieri rule. Therefore it suffices to establish the following claim:
\end{proof}

\begin{clm}\label{C}
Suppose $(n_1,n_2,n_3,n_4)\in[r\ell]^4$. Then there are Young diagrams $\lambda_i$, $i=1,\dots,4$  fitting into boxes of size  $r\times \ell$, so that $|\lambda_i|=n_i$, and
$$\deg \mathbb{V}_{\sL_{r+1}, (\lambda_1,\lambda_2,\lambda_3,\lambda_4),\ell}>0.$$
\end{clm}

\begin{proof}(of Claim \ref{C})

We will do this by induction on $r$. The weights $\lambda_j$'s will be such Theorem \ref{mon2} is applicable. So in addition to $\vec{\lambda}$ we will have subsets $A_i=\{\alpha_i<\beta_i\}\subseteq [r+1]$, $i=1,\dots,4$ and  associated conformal blocks bundles $\Bbb{V}_{\sL_{2},\vec{\mu},\ell}$ and $\Bbb{V}_{\sL_{r-1},\vec{\nu},\ell}$. This data will be such that conditions (a), (b) and (c) of Theorem \ref{mon2} hold, with $\delta=\ell+1$.
For $r=1$, $\sum n_i =2(\ell+1)$ and $0<n_i\leq \ell$, so any choice of $\lambda_i$ will work (use Fakhruddin's result that
critical level $\sL_2$ conformal blocks divisors are non-zero).

Assume that the claim holds for $r$ and prove it for $r+1\geq 2$ as follows. Let $m_1,m_2,m_3,m_4\in [(r+1)\ell]$ be positive integers which sum to
$(r+2)(\ell+1)$.

 We get $(n_1, \dots,n_4)\in[r\ell]^4$ and $(q_1,\dots,q_4)\in[\ell]^4$ from $(m_1, \dots, m_4)$ by applying Lemma \ref{abracadabra} below. Apply the claim (with the stronger burden of induction) for $r$ with data $n_1,\dots,n_4$, and obtain the data $\vec{\lambda}$, $\vec{\mu}$, $\vec{\nu}$ etc. Now add on a row of size
$q_i$ to $\lambda_i$ and get a new Young diagram $\lambda'_i$ (and permute rows so that one gets a legitimate Young diagram).  The old $\mu_i$ corresponds to rows $(\alpha'_i<\beta_i')$ of $\lambda_i'$. The new $\lambda_i'$ satisfies our requirement by using Theorem \ref{mon2} and Proposition
\ref{Phish}. Note that Proposition \ref{Phish} is applied to the $n$-tuple of weights of $\sL_r$ obtained by adding  rows of sizes $q_i$ to $\nu_i$.
(so, $a= r-1$ and $b=1$).

\end{proof}

\begin{lemma}\label{abracadabra}
It is possible to write $m_i=n_i+q_i, i=1,\dots,4, m_i,q_i\in \Bbb{Z}$ with $0<n_i\leq r\ell $ and $0\leq q_i\leq \ell$  and $\sum n_i= (r+1)(\ell+1)$ (so $\sum q_i =\ell+1$).
\end{lemma}
\begin{proof}First note that not more than two of the $m_i$ can be one since $(r+1)\ell +3< (r+2)(\ell+1)= r+ 2\ell +r\ell +2$.
\begin{itemize}
\item If $m_i$ are all $\leq r\ell$: Write $m_i=1+\delta_i$. Then $\sum\delta_i= (r+2)(\ell+1)-4\geq \ell+1$, since
$r+\ell+r\ell\geq 3$. There are at least two of the $m_i$ from which we may subtract, so we may restrict
$q_i$ to be $\leq \ell$.
\item If $m_1=r\ell +k, k>0$, $1\leq m_2,m_3,m_4\leq r\ell$, suppose $m_2>1$. Then, we may take  $q_1=\ell, q_2=1,q_3=0,q_4=0$.
\item If $m_1=r\ell+k_1$, $m_2=r\ell +k_2$ are $> r\ell$, and $1<m_3, m_4\leq r\ell$ (up to-reordering).
Then $k_1+k_2\leq (r+2)(\ell+1)-2r\ell-2 = r+2\ell-r\ell\leq \ell+1$ since  $r+\ell -r\ell -1= -(r-1)(\ell-1)\leq 0$. We take $q_1=k_1, q_2= (\ell+1)-k_1$ and $q_3=q_4=0$.
\item If $m_1=r\ell +k_1$, $m_2=r\ell +k_2$ and $m_3=r\ell +k_3$ are $> r\ell$ and $1<m_4\leq r\ell$. Then
$k_1+k_2+k_3\leq (r+2)(\ell+1)-3r\ell-1= [(r+2)(\ell+1) -2r\ell -2] +[1-r\ell]\leq \ell +1+0=\ell+1$. So we may set $q_1=k_1$, $q_2=k_2$, $ q_3=(\ell+1)-(k_1+k_2)$ and $q_4=0$.
\item If $m_i>r\ell$ for all $i$, then $4r\ell <(r+2)(\ell+1)= r\ell +2\ell +r +2$. Hence,
 $3r\ell-r-2\ell -r-2<0$  hence that $(r-1)(\ell-2) -4 +2r\ell<0$. If $\ell\geq 2$, then this cannot happen. If $\ell=1$, then
we get $3r-r-2-2<0$ or that $r<2$ and hence $r=\ell=1$. Writing $m_i= 1+\delta_i$, we see that $\sum \delta_i = 6-4=2$ but $\delta_i>0$ by assumption, so this case cannot happen.

\end{itemize}
\end{proof}

\subsection{The maps given by $\mathbb{D}=\mathbb{D}_{\sL_{r+1},\omega_1^n,\ell}$, $n=(r+1)(\ell+1)$}




\subsubsection{Hassett Contractions}\label{Hassett}
Consider
an $n$-tuple  $\mathcal{A}=\{a_1,\ldots,a_n\}$, with $a_i \in \mathbb{Q}$, $0 < a_i \le 1$, such that $\sum_i a_i >2$.   In \cite{HassettWeighted}, Hassett introduced moduli spaces $\overline{\operatorname{M}}_{0,\mathcal{A}}$, parameterizing families of stable weighted pointed rational curves $(C, (p_1,\ldots,p_n))$
such that
\begin{enumerate}
\item $C$ is nodal away from its marked points $p_i$;
\item $\sum_{j\in J} a_i \le 1$, if the marked points $\{p_j : j\in J\}$ coincide; and
\item If $C'$ is an irreducible component of $C$ then
$$\sum_{p_i \in C'} a_i + \mbox{ number of nodes on } C' > 2.$$
\end{enumerate}

These Hassett spaces  $\overline{\operatorname{M}}_{0,\mathcal{A}}$ receive birational morphisms $\rho_{\mathcal{A}}$ from
$\overline{\operatorname{M}}_{0,n}$ that are characterized entirely by which F-Curves (cf. Def. \ref{FCurve}) they contract.

\begin{definition/lemma}\label{rho}For any Hassett space $\overline{\operatorname{M}}_{0,\mathcal{A}}$, with  $\mathcal{A}=\{a_1,\ldots,a_n\}$, there are birational morphisms
$\rho_{\mathcal{A}}: \overline{\operatorname{M}}_{0,n} \longrightarrow  \overline{\operatorname{M}}_{0,\mathcal{A}}$,
which contract all  $\operatorname{F}$-curves $F(N_1,N_2,N_3,N_4)$ satisfying:
$$\sum_{i\in N_1 \cup N_2 \cup N_3} a_i \le 1,$$
and no others, where without loss of generality, the leg $N_4$ carries the most weight.
\end{definition/lemma}

\subsubsection{The maps given by $\mathbb{D}_{\sL_{r+1},\omega_1^n,\ell}$ contract same curves as Hassett Contractions}
Recall from \eqref{mapG} that conformal block divisors arise from suitable morphisms $\phi_{\Bbb{D}}$ from $\overline{\operatorname{M}}_{0,n}$ to projective spaces.

\begin{corollary}\label{Complete}
Suppose that $r>1$ and $\ell>1$. Put $n=(r+1)(\ell+1)$,  $\mathbb{D}=\mathbb{D}_{\sL_{r+1},\omega_1^{n},\ell}$, and $\mathcal{A}=(a_1,\ldots,a_n)$, with $a_i=\frac{1}{r+\ell}$.
Then  the maps $\phi_{\mathbb{D}}$ and $\rho_{\mathcal{A}}$  contract the same $\operatorname{F}$-curves.
\end{corollary}

\begin{proof}
This follows from  \cite{BGMA}, Theorem 6.2 and Proposition \ref{ExactNonzeroResult}.
\end{proof}

\subsubsection{Images of maps given by $\mathbb{D}_{\sL_{r+1},\omega_1^n,\ell}$}\label{magei}

Since the weights are symmetric the non-zero nef divisor $\mathbb{D}_{\sL_{r+1},\omega_1^n,\ell}$ is big, and so the corresponding morphism $\phi_{\mathbb{D}}: \overline{\operatorname{M}}_{0,n}\to \Bbb{P}$ is birational onto its image (\cite{KeelMcKernan,GibneyCompositio}).
By \cite{BGMA}, $\phi_{\mathbb{D}}$ factors via $\rho_{\mathcal{A}}$. Let $Z$ be the image of $\phi_{\mathbb{D}}$ and $\widetilde{Z}$ its normalization. We therefore find a
regular birational  morphism $p:\overline{\operatorname{M}}_{0,\mathcal{A}}\to \widetilde{Z}$. The $S_n$ invariant F-conjecture implies that $p$ is an isomorphism: In the lemma below the weights
$\mathcal{A}$ are arbitrary, and not necessarily symmetric.
\begin{lemma}\label{home}
Suppose $f:\overline{\operatorname{M}}_{0,n}\to \Bbb{P}^m$ is a morphism with image $Y$. Assume that $f$ contracts the same $F$-curves as the morphism
$\rho_{\mathcal{A}}: \overline{\operatorname{M}}_{0,n}\to \overline{\operatorname{M}}_{0,\mathcal{A}}$. Then, there is a morphism
$\pi:\overline{\operatorname{M}}_{0,\mathcal{A}}\to \widetilde{Y}$, where $\widetilde{Y}$ is the normalization of $Y$, which is an isomorphism if the $F$-conjecture holds (in particular, $f$ is birational on to its image).
\end{lemma}
\begin{proof}(Standard)
By Proposition 4.6 in \cite{Fakh}, $f$ factors through $\rho_{\mathcal{A}}$. There is therefore  a natural morphism $\overline{\operatorname{M}}_{0,\mathcal{A}}\to Y$, which factors through the normalization of $Y$, and hence a morphism $\pi$ as in the statement of the lemma. We need only show that $\pi$ does not contract any curves.
If $C$ is a contracted curve, we can write it as an image of a curve $C'$ in $\overline{\operatorname{M}}_{0,n}$ (this is possible because $\rho_{\mathcal{A}}$ is birational). Write $C'$
as a positive sum $\sum a_i F_i$ of $F$ curves, in the cone of curves, assuming the F-conjecture (if the weights are symmetric, by averaging we need only the $S_n$ invariant $F$-conjecture). Each $F_i$ is contracted by $f$,
and  is hence contracted by $\rho_{\mathcal{A}}$, this gives $C=0$ in the cone of curves. Therefore, $\pi$ is an isomorphism.
\end{proof}
\begin{remark}\label{remarkF}
It is therefore interesting to look for other (not necessarily symmetric) conformal block divisors $\Bbb{D}$ which contract the same curves as a suitable $\rho_{\mathcal{A}}$ (as in the main series of examples of this section). The F-conjecture implies, by Lemma \ref{home}, that such a $\Bbb{D}$ gives a birational morphism, with the normalization of image isomorphic to $\overline{\operatorname{M}}_{0,\mathcal{A}}$. Therefore strong vanishing and non-vanishing results (not the exact numerical values of classes) may lead to situations where the F-conjecture becomes applicable.
\end{remark}
\begin{remark}
The $\op{S}_n$-invariant $\op{F}$-Conjecture is known to hold for $n\le 24$ \cite{GibneyCompositio}, and so in this range one knows that the conformal blocks divisors $\Bbb{D}=\mathbb{D}_{\sL_{r+1}, \omega_1^n, \ell}$, for $n=(r+1)(\ell+1)$ give the maps from $\ovop{M}_{0,n}$ to the Hassett spaces $\ovop{M}_{0,\mathcal{A}}$.    In order to remove the dependence on the $\op{F}$-Conjecture, one could exhibit an ample divisor $D$ on $\overline{\operatorname{M}}_{0,\mathcal{A}}$ such that
the pull back of $D$ to $\overline{\operatorname{M}}_{0,n}$ is equal to a multiple of $\Bbb{D}$.  To do this one could show that the degree of the divisors $\mathbb{D}$ on a basis of $\op{F}$-Curves is equal to a multiple of the degree of the pullback of $D$ on those same curves.
In \cite{KiemMoon} and \cite{HBThesis}, it  has been shown that $\overline{\operatorname{M}}_{0,\mathcal{A}}$ can be constructed
as a GIT quotient by $\operatorname{SL}(2)$, and so canonical choices of such $D$ exist.
Moreover, in \cite[Proposition 5.6]{KiemMoon}, Kiem and Moon give a range of ample divisors on $\ovop{M}_{0,\mathcal{A}}$ and in \cite[Lemma 5.3]{KiemMoon}, they give the formulas for pulling those divisors back to
$\ovop{M}_{0,n}$.  One could try to solve for the parameter $\alpha$ in the given range which may determine the divisor $D$.  It would be enough to check degrees on those $\op{F}$-Curves of the form $F_{1,1,i,n-2-i}$ which form a basis for the $S_n$-invariant $1$-cycles.
\end{remark}

\section{The critical level divisors $\mathbb{D}_{\vec{\alpha}}=c_1\mathbb{V}_{\sL_{r+1}, \{\alpha_1\omega_1, \alpha_2 \omega_1, \ldots, \alpha_n \omega_1\}, \ell}$,  $n=2(r+1)$}\label{Kostka}
In this section, we show how one can  apply Theorem \ref{mon1} together with results from \cite{BGMA} to prove that divisors of the form
 $$\mathbb{D}_{\vec{\alpha}}=c_1(\mathbb{V}_{\sL_{r+1}, \{\alpha_1\omega_1, \alpha_2 \omega_1, \ldots, \alpha_n \omega_1\}, \ell}), \ \ \sum_{i=1}^n\alpha_i = (r+1)(\ell+1), \ \ n=2(r+1),\ \ell\geq 3$$
 are nonzero.

 \subsection{Nonvanishing} Since $\ell$ is the critical level for the pair $(r+1, \{\alpha_1 \omega_1, \ldots, \alpha_n \omega_1\})$, by \cite[Proposition 1.6 (b)]{BGMA} one has the critical level partner divisors:
$$\mathbb{D}_{\sL_{r+1}, (\alpha_1\omega_1, \alpha_2 \omega_1, \ldots, \alpha_n \omega_1), \ell}=\mathbb{D}_{\sL_{\ell+1}, \{\omega_{\alpha_1}, \omega_{\alpha_2} , \ldots, \omega_{\alpha_n}\},r},$$
the latter of which is at the theta level since
$$\theta(\sL_{\ell+1},  \{\omega_{\alpha_1}, \omega_{\alpha_2} , \ldots, \omega_{\alpha_n}\}) = -1 + \frac{n}{2}=-1+ \frac{1(r+1)}{2} = r.$$
In particular, we can apply Theorem \ref{mon1} to $\mathbb{D}_{\sL_{\ell+1}, \{\omega_{\alpha_1}, \omega_{\alpha_2} , \ldots, \omega_{\alpha_n}\},r},$ argue that the divisor is nonzero.
The associated auxiliary bundles are
 $$\mathbb{V}_{\sL_{2}, \omega_1^{2(r+1)}, r}, \mbox{ and } \mathbb{V}_{\sL_{\ell-1}, \{\omega_{\alpha_1-1},  \ldots, \omega_{\alpha_n-1}\},r}.$$
 In particular, by  Theorem \ref{mon1}, $\mathbb{D}_{\sL_{\ell+1}, \{\omega_{\alpha_1}, \omega_{\alpha_2} , \ldots, \omega_{\alpha_n}\},r}$ will be nonzero as long as both auxiliary bundles have positive rank.
The first bundle can be easily shown to have rank one using Witten's Dictionary.   Since $\sum_{i=1}^n (\alpha_i-1) = (r+1)(\ell-1)$ it follows from \cite{Fakh} that $\op{rk}\mathbb{V}_{\sL_{\ell-1}, \{\omega_{\alpha_1-1}, \omega_{\alpha_2-1} , \ldots, \omega_{\alpha_n-1}\},1}=1$ and hence $\op{rk}\mathbb{V}_{\sL_{\ell-1}, \{\omega_{\alpha_1-1}, \omega_{\alpha_2-1} , \ldots, \omega_{\alpha_n-1}\},r}\geq 1.$


\subsection{Associated maps}
By \cite[Proposition 5.3]{BGMA},  the $\mathbb{D}_{\vec{\alpha}}=\mathbb{D}_{\sL_{r+1}, (\alpha_1\omega_1, \alpha_2 \omega_1, \ldots, \alpha_n \omega_1), \ell}$ have zero intersection with all the F-Curves which are contracted by the Hassett maps $$\rho_{\mathcal{A}}: \ovop{M}_{0,n} \longrightarrow \ovop{M}_{0,\mathcal{A}}, \mbox{ where } \ \ \mathcal{A}=(a_1,\ldots,a_n), \mbox{  and  } a_i =\frac{\alpha_i}{r+\ell}.$$   In particular, the morphisms given by these divisors factor through $\rho_{\mathcal{A}}$.

As we'll see next in \ref{NotHassett}, there are choices of $\vec{\alpha}$ for which the divisors $\mathbb{D}_{\vec{\alpha}}$ contract more curves that the morphisms $\rho_{\mathcal{A}}$, and so while their associated maps factor through the Hassett reduction morphisms, their images are not isomorphic to $\ovop{M}_{0,\mathcal{A}}$.

\subsection{A particular example.}\label{NotHassett}  For a  simple example, consider $\mathbb{D}=c_1\mathbb{V}_{\sL_{3}, (3\omega_1)^6, 5}$.  Here $r=2$, $\ell=5$, and $n=6$.  Since $\mathbb{D}$ is a critical level divisor, we have that
$$c_1(\mathbb{V}_{\sL_{3}, (3\omega_1)^6, 5})=c_1(\mathbb{V}_{\sL_{6}, (\omega_3)^6, 2}).$$
And as above, the theta and critical levels coincide for the divisor $c_1(\mathbb{V}_{\sL_{6}, (\omega_3)^6, 2})$, and so we may apply Theorem \ref{mon1} to check it is nonzero.  The two auxiliary bundles are
$\mathbb{V}_{\sL_{2}, (\omega_1)^6, 2}$, which has rank one, and $\mathbb{V}_{\sL_{4}, (\omega_2)^6, 2}$, which has rank $11$, as is easily checked by a computation \cite{Swin}.    One also checks that
the class of the divisor is $2B_2+3B_3$, where $B_i$ is the sum of boundary classes of $\ovop{M}_{0,6}$ indexed by sets of size $i$.  The Hassett contraction
 $\rho_{\mathcal{A}}: \ovop{M}_{0,n} \longrightarrow \ovop{M}_{0,\mathcal{A}}$ where $\mathcal{A}=(a_1,\ldots,a_n)$, and $a_i =\frac{3}{7}$,
 will not contract any $F$-curves, on the other hand, it is easy to see that $2B_2+3B_3$ contracts the curve $F_{1,1,1,3}$.

 \subsection{The example $r=1$ and $\vec{\alpha}=(1,1,1,1)$.}\label{1} In case $r=1$,  and $\vec{\alpha}=(1,1,1,1)$,
we have the divisor $\mathbb{D}_{\sL_{2}, \omega_1^4, \ell}=2(\delta_{12}+\delta_{13}+\delta_{14})$ on $\ovop{M}_{0,4}$.  This
is also the first member of a related family of divisors $\mathbb{D}_{\sL_{2},\omega_1^n,\ell}$ on $\ovop{M}_{0,n}$, for $n=2k$ studied in \cite{gjms}.  Like our family, $\mathbb{D}_{\sL_{2},\omega_1^n,\ell}$ was shown to give rise to maps which factor through maps to Hassett spaces but contract more curves than the $\rho_{\mathcal{A}}$.

\section{Appendix}\label{appendo}
\subsection{Flag varieties and line bundles}\label{verlindesection}


\begin{defi}\label{classic}
Let $V$ be a vector space of rank $r+1$.
\begin{enumerate}
\item
A complete flag on $V$ is a filtration of $V$ by vector subspaces
$$F_{\bull}:0\subsetneq F_1\subsetneq F_2\subsetneq\dots\subsetneq F_{r+1}=V.$$
The space of complete flags on $V$ is denoted by $\Fl(V)$.
\item The determinant line $\wedge^{r+1} V$ is denoted by $\det V$.
\item For non-negative integers $\nu_1,\dots,\nu_{r+1}$, define a Young diagram $\lambda=(\lambda^{(1)},\dots,\lambda^{(r+1)})$ by
$$\lambda^{(a)}=\nu_a+\nu_{a+1}+\dots+\nu_{r+1}$$
and a line bundle $\ml_{\lambda}$ over $\Fl(V)$, whose fiber over $F_{\bull}$ is  $$\ml_{\lambda}(V,F_{\bull})=\ml_{\lambda}(F_{\bull})=(\det F_1)^{-\nu_1}\tensor\dots\tensor (\det F_{r+1})^{-\nu_{r+1}}.$$

\item   We fix a collection of $n$ distinct and ordered points $S=\{p_1,\dots,p_n\}\subseteq\pone$, and
for a vector bundle $\mw$ on $\pone$, define $\Fl_S(\mw)=\prod_{p\in \mpp}\Fl(\mw_p).$
If $\me\in \Fl_S(\mw)$, we will assume that it is written in the form
$\me=(E^p_{\bull}\mid p\in \mpp)$.
\end{enumerate}
\end{defi}

\subsection{Conformal blocks as generalized theta functions}

Associated to the data  $\vec{\lambda}=(\lambda_1,\dots,\lambda_n)\in P_{\ell}(\sL_{r+1})$, we can form a line bundle $\mathcal{P}(\sL_{r+1},\tilde{\ell},\vec{\lambda})$ on $\parbun_{r+1}$. The fiber over a point $(\mv, \mf,\gamma)$ is a tensor product
$$D(\mv)^{\tilde{\ell}}\tensor\tensor_{i=1}^n \ml_{\lambda_{i}}(\mv_{p_i},F^{p_i}_{\bull}),$$
where $D(\mv)$ is the determinant of cohomology of $\mv$ i.e., the line $\det H^1(\mathbb{P}^1, \mv)\tensor\det H^0(\mathbb{P}^1, \mv)^*$ and the lines $\ml(\mv_{p_i},F^{p_i}_{\bull})$ are as in Definition ~\ref{classic}.

It is known that the space of generalized theta functions is canonically identified (up-to scalars) with the dual of the space of conformal blocks (see the survey \cite{sorger}, and \cite{pauly}). Let $\xpp=(\pone,p_1,\dots,p_n)\in \operatorname{M}_{0,n}$.
\begin{equation}\label{verlinde}
 H^0(\parbun_{r+1},  \mathcal{P}(\sL_{r+1},\tilde{\ell},\vec{\lambda}))\leto{\sim} (\Bbb{V}_{\sL_{r+1},\vec{\lambda},\tilde{\ell}})_{\xpp}^*.
\end{equation}
\subsection{Non zero sections of conformal blocks bundles}\label{openeye2}A bundle $\mathcal{V}=\oplus_{i=1}^r\mathcal{O}_{\mathbb{P}^1}(a_i)$ on $\mathbb{P}^1$ is said to be evenly split if $|a_i-a_j|\leq 1$ for all $1\leq i,j\leq r$. Up to isomorphism, there is an unique evenly split bundle on $\mathbb{P}^1$ of a fixed degree.

Consider  $x=(\pone,p_1,\dots,p_n)\in \operatorname{M}_{0,n}$. The space
$\vr|_x^*$ is identified with $H^0(\parbun_{r+1},\mathcal{P}_{r+1})$ where $\mathcal{P}_{r+1}$ is the line bundle $\mathcal{P}( \sL_{r+1},{\ell},\vec{\lambda})$.
\begin{lemma}\label{closedeye}
The following are equivalent (see \cite{b4})
\begin{enumerate}
\item $\rk\vr\neq 0$.
\item There exist  vector bundles $\mv$ and $\mq$ of degrees $0$ and  $-D$ respectively, and ranks $r+1$ and rank $\ell$ respectively, and $\mf\in\Fl_S(\mv)$ and $\mgg\in \Fl_S(\mq)$ so that the vector space
\begin{equation}\label{openeye}
\{\phi_p\in\operatorname{Hom}(\mv,\mq)\mid \phi(F^p_a)\subseteq G^p_{\ell-\lambda_i^{(a)}}, p=p_i\in S, a\in[r+1]\}
\end{equation}
is zero.
\end{enumerate}
\end{lemma}
 This is a special case of Proposition 5.5 from \cite{b4}.
For a fixed $y=(\mq,\mgg)$ we may think of \eqref{openeye} as defining a section  $s_y\in H^0(\parbun_{r+1},\mathcal{P}_{r+1})$ (see \cite{BGMA} for a construction, also see \cite{PB,RO}). 
The section $s_y$ does not vanish at $(\mv,\mf,\gamma)\in\parbun_{r+1}$ if and only if \eqref{openeye} is zero.

The following is used in the proof of Proposition \ref{ExactNonzeroResult}.

\begin{proposition}\label{Phish}
Suppose $\vec{\nu}$ is an $n$-tuple of dominant integral weights in  $P_{\ell}(\sL_{\tr})$ ($\vec{\nu}$ may not be normalized).
Suppose $\tr=a+b$ with $a$ and $b$ positive integers. Let $A=(A_1,\dots,A_n)$ and $B=(B_1,\dots,B_n)$ be $n$-tuples of subsets of $[\tr]$ such that $|A_i|=a$ and $|B_i|=b$, $[\tr]=A_i\cup B_i,\  i=1,\dots,n$.
Let $\vec{\nu}_{A}=(\nu_{1,A_1},\dots,\nu_{n,A_n})$ be the $n$-tuple of $a\times \ell$ Young diagrams formed by taking the  $A_i$-rows of $\nu_i$ for each $i$ (similarly
 $\vec{\nu}_B$).
 Suppose
 \begin{enumerate}
 \item
 \begin{equation}\frac{1}{\tr}\sum_{i=1}^n |\nu_i|=\frac{1}{a}\sum_{i=1}^n  | \nu_{i,A_i}|=\frac{1}{b}\sum_{i=1}^n |\nu_{i,B_i}|=   \delta\in \Bbb{Z}.
 \end{equation}
\item  If $a>1$, then $\rk \Bbb{V}_{\sL_{a},\ell,\vec{\nu}_A}\neq 0$.
\item  If $b>1$,  then $\rk \Bbb{V}_{\sL_{b},\ell,\vec{\nu}_B}\neq 0$.
\end{enumerate}
 Then $\rk\Bbb{V}_{\sL_{\tr},\ell,\vec{\nu}}\neq 0$.
\end{proposition}

\begin{proof}
We deduce the proof
from (one of the forms of) the quantum generalization of the Horn conjecture  \cite[Proposition 3.4]{b4}.

For every dominant integral weight $\lambda$ of $\sL_{r+1}$
define a diagonal matrix $\alpha(\lambda,\ell)=\alpha_{ij}$ in the special unitary group $\operatorname{SU}(r+1)$ with diagonal entries
$$\alpha_{aa}= c^{-1}\exp(\frac{2\pi i \lambda^{(a)}}{\ell}), \ c=\exp(\frac{2\pi i |\lambda|}{\ell (r+1)}).$$

Consider a conformal blocks bundle $\vr$ such that $(r+1)$ divides  $\sum|\lambda_i|$.  We now use the proposition below (by forming suitable direct sums): If $A_i$ in $\operatorname{U}(a)$ and $B_i\in\operatorname{U}(b)$, with $A_1A_2\dots A_n=\gamma \operatorname{Id}$ and $B_1B_2\dots B_n=\gamma\operatorname{Id}, $ then we form matrices $C_i\in \operatorname{U}(a+b)$ with $C_1C_2\dots C_n=\gamma \operatorname{Id}$ by direct sum.  If $b=1$, we let $B_i=(b_i)$ be the $1\times 1$ matrices with $b_i=\exp(\frac{2\pi i \nu^{(1)}_{i,B_i}}{\ell})$.
\end{proof}
\begin{proposition}\cite[Proposition 3.4]{b4}\label{hornp}
The following are equivalent:
\begin{enumerate}
\item $\rk \vr\neq 0$.
\item There exist matrices $A_i\in\operatorname{SU}(r+1)$ with $A_1A_2\dots A_n=\operatorname{Id}$ where each $A_i$ is conjugate to  $\alpha(\lambda_i,\ell)$.
\item There exist matrices $A_i$ in the unitary group $\operatorname{U}(r+1)$ with $A_1A_2\dots A_n=\gamma \operatorname{Id}$ where  each $A_i$ is conjugate to a diagonal matrix with entries $\exp(\frac{2\pi i \lambda^{(a)}_i}{\ell})$, and $\gamma=\exp(\frac{2\pi i \sum_i |\lambda_i|}{\ell (r+1)})$.
\end{enumerate}
\end{proposition}

\vspace{0.4 in}

\noindent
P.B.: Department of Mathematics, University of North Carolina, Chapel Hill, NC 27599\\
email: belkale@email.unc.edu

\vspace{0.2 cm}

\noindent
A.G.: Department of Mathematics, University of Georgia, Athens, GA 30602\\
email: agibney@math.uga.edu

\vspace{0.2 cm}

\noindent
S.M.: Department of Mathematics, University of Maryland,
College Park, MD 20742\\
email: swarnava@umd.edu

\end{document}